\documentclass[11pt]{jdg-p}
\usepackage{amscd, amssymb, latexsym, amsmath}

\newtheorem{thm}{Theorem}
\newtheorem{cor}{Corollary}
\newtheorem{pro}{Proposition}
\newtheorem{lem}{Lemma}
\newtheorem{dfn}{Definition}

\newtheorem{rem}{Remark}

\numberwithin{equation}{section}
\newcommand{\cp}{{\mathbb C}{\mathbb P}}
\newcommand{\R}{\mathbb R}

\newcommand{\comments}[1]{}

\begin{document}
\title[Deforming symplectomorphisms]{Deforming symplectomorphisms of \\complex projective spaces\\ by the mean curvature flow}
\author{Ivana Medo\v s and Mu-Tao Wang}
\thanks{The authors are partially supported by National Science
Foundation Grant DMS 0605115 and DMS 0904281.}
\date{July 15, 2009, this version January 25, 2011}
\maketitle
\begin{abstract}
We apply the mean curvature flow to deform symplectomorphisms of $\mathbb{CP}^n$. In particular, we prove that, for each dimension $n$, there exists a constant $\Lambda$, explicitly computable, such that any $\Lambda$-pinched symplectomorphism of $\mathbb{CP}^n$ is symplectically isotopic to a biholomorphic isometry.
\end{abstract}
\section{Introduction}

It was proposed in \cite{wa4} to use the mean curvature flow to study the structure of the symplectomorphism group of a symplectic manifold $(M, \omega)$.  Consider the graph of a symplectomorphism $f:M\rightarrow M$ as an embedded submanifold $\Sigma=\{(x, f(x))\,\,|\,\, x\in M\}$ of the product manifold $M\times M$. $\Sigma$ can be viewed as a Lagrangian submanifold with respect to the symplectic structure $\pi_1^*\omega-\pi_2^*\omega$ on $M\times M$ where $\pi_i$ is the projection from $M\times M$ to the $i$-th factor, $i=1, 2$. Suppose that $M$ is endowed with a compatible K\"ahler metric such that $\omega$ is the K\"ahler form. The volume of $\Sigma$ with respect to the product metric naturally defines a function on the symplectomorphism group of $M$ which is symmetric with respect to the inverse operation $f\mapsto f^{-1}$. This provides a variational approach to study the topology of this infinite dimensional group. The critical point of the volume function corresponds to minimal Lagrangian submanifolds and the mean curvature flow is the negative gradient flow. By Smoczyk \cite{sm1}, it is known that being Lagrangian is preserved by the mean curvature flow when $M$ is equipped with a K\"ahler-Einstein metric. Therefore, if $\Sigma$ remains graphical along the mean curvature flow, the flow in turn gives a symplectic isotopy of $f$.

In this article, we apply this idea to the complex projective space $\mathbb{CP}^n$ with the Fubini-Study metric and prove that a pinched symplectomorphism (see Definition~\ref{pinching}) is symplectically isotopic to a biholomorphic isometry along the mean curvature flow.

Denote by $g$ and $\omega$ the Fubini-Study metric and the associated K\"ahler form on $\mathbb{CP}^n$, respectively. Recall that a diffeomorphism $f$ of $\mathbb{CP}^n$ is a symplectomorphism if $f^*\omega=\omega$.

\begin{dfn}\label{pinching}
Let $\Lambda$  be a constant $\geq 1$. A symplectomorphism $f$ of $\mathbb{CP}^n$ is said to be $\Lambda$-pinched if \begin{equation}\frac{1}{\Lambda^2} g\leq f^* g\leq \Lambda^2 g.\end{equation}
\end{dfn}

  The precise statement of the pinching theorem is the following.

\begin{thm}\label{theorem}

For each positive integer $n$ there exists a constant $\Lambda(n)>1$, such that, if $f:\cp^n\rightarrow \cp^n$ is a $\Lambda$-pinched symplectomorphism for some $1<\Lambda<\Lambda(n)$, then:

1) The mean curvature flow $\Sigma_t$ of the graph of $f$ in $\cp^n\times \cp^n$ exists smoothly for all $t\geq 0$.

2) $\Sigma_t$ is the graph of a symplectomorphism $f_t$ for each $t\geq 0$.

3)$f_t$ converges smoothly to a biholomorphic isometry of $\cp^n$ as $t\rightarrow \infty$.

\end{thm}

The mean curvature flow forms a smooth one-parameter family of symplectomorphisms or a symplectic isotopy. Therefore the following holds.
\begin{cor}
For each positive integer $n$, there exists a constant $\Lambda(n)$, such that if $f$ is a $\Lambda$-pinched symplectomorphism of $\cp^n$ for some $1<\Lambda<\Lambda(n)$, then $f$ is symplectically isotopic to a biholomorphic isometry.
\end{cor}

This theorem generalizes a previous theorem of the second author for Riemann surfaces in which no pinching condition is required.

\begin{thm}\label{surface_case} \cite{wa2, wa5}
Let $(\Sigma^{1}, g_1, \omega_1)$ and $(\Sigma^{2},g_2, \omega_2)$ be two
diffeomorphic compact Riemann surfaces with Riemannian metrics $g_1$ and $g_2$ of the same constant
curvature $c$. Suppose $\Sigma$ is the graph of a
symplectomorphism $f:\Sigma^{1}\rightarrow \Sigma^{2}$
and $\Sigma_t$ is the mean curvature flow in the product space $\Sigma^{1}\times \Sigma^{2}$ with initial surface
$\Sigma_0=\Sigma$. Then $\Sigma_t$ remains the graph of a
symplectomorphism $f_t$ along the mean curvature flow. The flow
exists smoothly for all time and $\Sigma_t$ converges smoothly to
a minimal Lagrangian submanifold as $t\rightarrow \infty$.
\end{thm}

In Theorem~\ref{surface_case}, the long time existence for any $c$ and the smooth convergence for $c>0$  were  proved in
\cite{wa2}. The smooth convergence for $c\leq 0 $ was established in Theorem 1.1 of \cite{wa5}. Using a different method, Smoczyk \cite{sm2} proved the theorem when $c\leq 0$ assuming an angle condition.  The existence of the limiting minimal Lagrangian surface was proved earlier using variational method by Schoen \cite{sc} (see also \cite{le}). In this case the symplectomorphism is indeed an area-preserving map. The boundary value problem for minimal area-preserving maps has been studied by Wolfson \cite{wo} and Brendle \cite{br}.

 A theorem of Smale states that the isometry group $SO(3)$ of $S^2$ is a continuous deformation retract of the oriented diffeomorphism group of $S^2=\cp^1$, and Theorem \ref{surface_case} gives a new proof of this theorem. The deformation retract provided by the mean curvature flow is indeed smooth. We are informed by Prof. McDuff that it was proved by Gromov \cite{gr} that the biholomorphic isometry group of $\cp^2$ is a deformation retract of its symplectomorphism group. It seems that no similar result is known for $\cp^n$ when $n>2$.

The proof is divided into several steps:

Step 1. We make several observations about singular values and singular vectors of symplectomorphisms. We also discuss the geometric properties of graphs of symplectomorphisms of K\"ahler-Einstein manifolds, as well as the setup of our problem. (see \S 2)

Step 2. We claim that $\Sigma_t$ remains the graph of a symplectomorphism $f_t$ as long as the flow exists smoothly. We study the evolution of the Jacobian of the projection map $\pi_1:\Sigma_t \rightarrow M$ (denoted by $*\Omega$) and prove that positivity is preserved by the maximum principle. This justifies the claim by the implicit function theorem. (see \S 3.1 and \S 3.2)

Step 3. We apply the blow up analysis to bound the second fundamental form of $\Sigma_t$ for each $t> 0$, and show that there is no finite time singularity. (see \S 3.3)

Step 4. We study the long time behavior of the evolution and use a comparison principle to show that the pinching condition is improved (by the curvature property of $\cp^n$) and the pull-back metric $f^*g$ is approaching $g$ as $t\rightarrow \infty$.

Step 5. We prove that the second fundamental form of $\Sigma_t$ is uniformly bounded in $t$ as $t\rightarrow \infty$. This gives the smooth convergence in the theorem.

Step 4 and 5 are done in \S 3.4.

The authors would like to thank Prof.~Dusa McDuff for particularly valuable discussions and suggestions. The second author would like to thank Mr. Bang Xiao for pointing out several typos in a previous version and an anonymous referee for suggestions in improving the presentation.

\section{Preliminaries}
\subsection{Singular values of symplectic linear maps between vector spaces}

Let $(V, g)$ and $(\tilde{V},\tilde{g})$ be $2n$-dimensional real inner product spaces, with almost complex structures $J$ and $\tilde{J}$, respectively, compatible with the corresponding inner products. Then $\omega(\cdot,\cdot)=g(J\cdot, \cdot)$, and $\tilde{\omega}=\tilde{g}(\tilde{J} \cdot, \cdot)$ are symplectic forms on $V$ and $\tilde{V}$. Recall that a linear map  $L: (V, \omega)\rightarrow (\tilde{V}, \tilde{\omega})$ is said to be symplectic if:
\begin{equation}\label{symplectomorphism}
\omega(u, v)=\tilde{\omega}(L(u), L(v))
\end{equation}
for any $u, v\in V$. In this context, the condition is equivalent to:
\begin{equation}\label{symplectomorphism2}
L^*\tilde{J}L=J,
\end{equation}
where $L^*:\tilde{V}\rightarrow V$ is the adjoint operator of $L$ with respect to the inner products on $\tilde{V}$ and $V$.

For such $L$, we define $E:V\rightarrow \tilde{V}$ to be the map $E=L[L^*L]^{-\frac{1}{2}}$. Since $L$ is an isomorphism, $L^*L$ is a positive definite self-adjoint automorphism of $V$ and the square root of $L^*L$ is well-defined.

\begin{lem}\label{QJcomm} $E$ is an isometry which intertwines with $J$ and $\tilde{J}$, i.e.
\[\tilde{J} E=EJ.\]
In other words, $E$ is a symplectic isometry.
\end{lem}
\begin{proof}
$E$ is an isometry since:
\begin{align*}
\tilde{g}(E u, E v)=\tilde{g}(L[L^*L]^{-\frac{1}{2}}u,L[L^*L]^{-\frac{1}{2}} v)&=g(L^*L[L^*L]^{-\frac{1}{2}} u, [L^*L]^{-\frac{1}{2}} v)\\
&=g([L^*L]^{\frac{1}{2}} u, [L^*L]^{-\frac{1}{2}} v)\\
&=g([L^*L]^{-\frac{1}{2}}[L^*L]^{\frac{1}{2}} u,  v)\\
&=g(u,v)
\end{align*}
 for any $u,v\in V$. Let $P=[L^*L]^{\frac{1}{2}}$, so that $E=LP^{-1}$. $-J P^{-1}J$ and $P$ are both positive definite ($-JP^{-1}J=J^{-1}P^{-1}J$ is positive definite since $P^{-1}$ is and since $J$ is an orthogonal operator), and, by the symplectic condition~\eqref{symplectomorphism2}, their squares are equal:
\begin{align*}
(-JP^{-1}J)^2&=-JL^{-1}(L^*)^{-1}J \\
&=-L^*\tilde{J}\tilde{J}L \\
&=P^2.
\end{align*}

It follows that $-J P^{-1}J=P$.
By using the symplectic condition $L^*\tilde{J}L=J$ and the fact that $P=L^*L P^{-1}$, we obtain the desired result:
\begin{align*}
-J P^{-1}J=P &\Rightarrow -J P^{-1}J=L^*L P^{-1} \\
&\Rightarrow -(L^*)^{-1}J P^{-1}J=L P^{-1} \\
&\Rightarrow -\tilde{J}L P^{-1}J=L P^{-1} \\
&\Rightarrow -\tilde{J}EJ=E.
\end{align*}

Finally, the last equality implies $E^*\tilde{J}E=J$
so $E$ is in fact a symplectic isometry (condition~\eqref{symplectomorphism2}).

\end{proof}

Let $(v_1,\ldots,v_{2n})$ be a basis of $V$ that diagonalizes $L^*L$. $L^*L$ is the positive definite matrix:
\[L^*L=\left(\begin{array}{ccccc}
\lambda_1^2 & 0 & \ldots & & 0 \\
0 & \lambda_2^2  &  &  & \\
\vdots &  & \ddots &  & \\
 & & & \lambda_{2n-1}^2 & \\
0 & & \ldots  & 0 & \lambda_{2n}^2 \\
\end{array} \right)\] with respect to this basis, for some $\lambda_i>0$, $i=1,\ldots,2n$.

Then, by construction, $L(v_i)=\lambda_iE(v_i)$; in other words:\\
\[L=\left(\begin{array}{ccccc}
\lambda_1 & 0 & \ldots & & 0 \\
0 & \lambda_2  &  &  & \\
\vdots &  & \ddots &  & \\
 & & & \lambda_{2n-1} & \\
0 & & \ldots  & 0 & \lambda_{2n} \\
\end{array} \right)\] with respect to the bases $(v_1,\ldots,v_{2n})$ and $(E(v_1),\ldots, E(v_{2n}))$, and thus $\lambda_i$ are the singular values of $L$.

\begin{lem}\label{lambdagJ} Let $\lambda_i$ be the singular values of $L$ and $v_i$ be the associated singular vectors, i.e. $L(v_i)=\lambda_iE(v_i)$. Then:
\[(\lambda_i\lambda_j-1)g(Jv_i, v_j)=0.\]
\end{lem}

\begin{proof} By the symplectic condition~\eqref{symplectomorphism} and Lemma~\ref{QJcomm}:
\begin{align*}
g(Jv_i, v_j)=&\tilde{g}(\tilde{J}L(v_i), L(v_j))=\lambda_i\lambda_j\tilde{g}(\tilde{J}E(v_i), E(v_j))\\
&=\lambda_i\lambda_j\tilde{g}(E(Jv_i),E(v_j))\\
&=\lambda_i\lambda_jg(Jv_i, v_j).
\end{align*}

\end{proof}

\begin{lem}\label{singvalpairs}
If $\alpha$ is a singular value of $L$, then so is $\frac{1}{\alpha}$. Moreover, if $V(\alpha)$ denotes the subspace of singular vectors corresponding to a singular value $\alpha$, then
\[\dim V(\alpha)=\dim V\left(\frac{1}{\alpha}\right),\]
and $J$ restricts to an isomorphism between $V(\alpha)$ and $V\left(\frac{1}{\alpha}\right)$.
\end{lem}
\begin{proof}
The first statement is a consequence of Lemma~\ref{lambdagJ}. Indeed, let $(v_1,\ldots,v_{2n})$ be the basis described in the lemma. Then for each $i\in\{1,\ldots,2n\}$ there exists some $j\in\{1,\ldots,2n\}$ such that $g( Jv_i, v_{j})\neq 0$ since $Jv_i$ is a nonzero vector. Then, by the lemma, it follows that $\lambda_i\lambda_{j}=1$.

The second statement is trivial if $\alpha=1$. Assume that $\alpha\neq 1$, and let $\dim V(\alpha)=k$, $\dim V\left(\frac{1}{\alpha}\right)=l$. By renumbering indexes, we may assume that $v_{1},\ldots,v_{k}$ span $V(\alpha)$ (so that $\lambda_{1}=\ldots=\lambda_{k}=\alpha$). We claim that $Jv_{1},\ldots, Jv_{k}$ belong to $V\left(\frac{1}{\alpha}\right)$. Fix any $1\leq i\leq k$ and consider $Jv_i$. Let $V'$ be the orthogonal complement of $V(\frac{1}{\alpha})$ such that $V=V(\frac{1}{\alpha})\oplus V'$. Take any $v_m\in V'$ for $1\leq m\leq 2n$, thus we have $Lv_m=\lambda_m v_m$ for $\lambda_m \not= \frac{1}{\alpha}$. Lemma \ref{lambdagJ} implies $g(Jv_i, v_m)=0$ for any such $v_m$, and therefore $Jv_i$ is in the orthogonal complement of $V'$, or $V(\frac{1}{\alpha})$ for each $i=1,\cdots k$.
Moreover, $Jv_{1},\ldots, Jv_{k}$ are linearly independent because $v_{1},\ldots,v_{k}$ are. It follows that $k\leq l$. The same argument applies to $V\left(\frac{1}{\alpha}\right)$ and it follows that $k\geq l$.

We conclude that $k=l$, and that $J$ restricts to an isomorphism from $V(\alpha)$ to $V\left(\frac{1}{\alpha}\right)$.

\end{proof}

\begin{rem}
The preceding lemma implies that $V$ splits into a direct sum of singular subspaces of the following form:
\begin{equation}\label{split}
V=V(1)^{k_0}\oplus V(\alpha_1)^{k_1}\oplus V\left(\frac{1}{\alpha_1}\right)^{k_1}\oplus\ldots \oplus V(\alpha_s)^{k_s}\oplus V\left(\frac{1}{\alpha_s}\right)^{k_s},
\end{equation}
where  $s$ is the total number of distinct singular values of $L$ greater than 1, $\alpha_i$ are distinct singular values of $L$ greater than 1, $i=1,\ldots,s$, and the superscripts represent dimension, $k_0\geq 0$ and $k_j> 0$ for $j=1,\ldots,s$.
\end{rem}

\begin{pro}\label{basis}
Let $L:(V,\omega)\rightarrow (\tilde{V},\tilde{\omega})$ be a symplectic linear map, where $V$ and $\tilde{V}$ are real vector spaces of dimension $2n$ equipped with almost complex structures $J$ and $\tilde{J}$ and inner products $g$ and $\tilde{g}$ compatible with the respective complex structures; and where $\omega=g(J\cdot,\cdot)$, $\tilde{\omega}=\tilde{g}(\tilde{J}\cdot,\cdot)$. Then there exists an orthonormal basis of $V$ with respect to which:
\begin{equation}\label{formofJ}
J=\left(\begin{array}{cccc}
0 & -1 & \ldots & 0 \\
1 & 0  & \ldots &  0\\
\vdots &  & \ddots &  \\
0 & \ldots & 0 & -1 \\
0 & \ldots  & 1 & 0 \\
\end{array} \right)
\end{equation}
and:
\begin{equation}\label{formofLstarL}
L^*L=\left(\begin{array}{ccccc}
\lambda_1^2 & 0 & \ldots & & 0 \\
0 & \lambda_2^2  &  &  & \\
\vdots &  & \ddots &  & \\
 & & & \lambda_{2n-1}^2 & \\
0 & & \ldots  & 0 & \lambda_{2n}^2 \\
\end{array} \right)
\end{equation}
where $\lambda_{2i-1}\lambda_{2i}=1$, for $i=1,\ldots,n$.
\end{pro}

\begin{proof}
Lemma~\ref{singvalpairs} and \eqref{split} imply that it is sufficient to find a basis satisfying~\eqref{formofJ} of the subspaces $V(\alpha)\oplus V(\frac{1}{\alpha})$  for each singular value $\alpha\neq 1$, as well as of $V(1)$ if $1$ is a singular value of $L$.

Assume that there is a singular value $\alpha\neq 1$, and let $k=\dim V(\alpha)$. We choose an arbitrary basis $u_1,\ldots,u_k$ of this space. Then $Ju_1,\ldots, Ju_k$ is a basis of $V(\frac{1}{\alpha})$. Putting these bases together provides a basis of $V(\alpha)\oplus V(\frac{1}{\alpha})$ satisfying~\eqref{formofJ}. Moreover, since $u_1,\ldots,u_k$ are singular vectors of $L$ with singular value $\alpha$, and $Ju_1,\ldots,Ju_k$ are singular values of $L$ with singular value $\frac{1}{\alpha}$, it follows that $(u_1, Ju_1, u_2, Ju_2,\ldots, u_k, Ju_k)$ is the desired basis.

If a singular value is equal to 1 (i.e. if $k_0>0$ in~\eqref{split}), any basis of $V(1)$ satisfying~\eqref{formofJ} suffices.

\end{proof}

Since the image of an orthonormal basis under an isometry is also an orthonormal basis, we obtain the following corollary.

\begin{cor}\label{imagebasis}
Let $E:V\rightarrow\tilde{V}$ be the isometry $E=L[L^*L]^{-\frac{1}{2}}$. If $(a_1,\ldots,a_{2n})$ is a basis of $V$ satisfying the properties of Proposition~\ref{basis}, and if $(\tilde{a}_1,\ldots,\tilde{a}_{2n})$ is the orthonormal basis $(E(a_1),\ldots,E(a_{2n}))$ of $\tilde{V}$, then:\\
(a) \[\tilde{J}=\left(\begin{array}{cccc}
0 & -1 & \ldots & 0 \\
1 & 0  & \ldots &  0\\
\vdots &  & \ddots &  \\
0 & \ldots & 0 & -1 \\
0 & \ldots  & 1 & 0 \\
\end{array} \right)\]
with respect to $(\tilde{a}_1,\ldots,\tilde{a}_{2n})$;\\
and:\\
(b) $L$ is diagonalized with respect to these bases, with diagonal values ordered in pairs whose product is 1:
\[L=\left(\begin{array}{ccccc}
\lambda_1 & 0 & \ldots & & 0 \\
0 & \lambda_2  &  &  & \\
\vdots &  & \ddots &  & \\
 & & & \lambda_{2n-1} & \\
0 & & \ldots  & 0 & \lambda_{2n} \\
\end{array} \right)\]
with $\lambda_{2i-1}\lambda_{2i}=1$, for $i=1,\ldots,n$.
\end{cor}

\begin{proof}
Part (a) follows from Proposition~\ref{basis} and Lemma~\ref{QJcomm}. Part (b) follows from the fact that $L(a_i)=\lambda_iE(a_{i})$.

\end{proof}

\subsection{Geometry of graphs of symplectomorphisms}

Let $\Sigma$ be the graph of a symplectomorphism $f:(M,\omega)\rightarrow(\tilde{M},\tilde{\omega})$ between K\"ahler-Einstein manifolds $(M,g, \omega)$ and $(\tilde{M},\tilde{g}, \tilde{\omega})$ of real dimension $2n$ and of the same scalar curvature. The product space $(M\times \tilde{M}, G=g\oplus \tilde{g})$ is thus a K\"ahler-Einstein manifold. We consider the evolution of $\Sigma \subset M\times \tilde{M}$ under the mean curvature flow.  If $J$ and $\tilde{J}$ are almost complex structures of $M$ and $\tilde{M}$, respectively, then $\mathcal{J}=J\oplus(-\tilde{J})$ defines an almost complex structure on $M\times \tilde{M}$ parallel with respect to $G$. Let $\Sigma_t$ be the mean curvature flow of $\Sigma$ in $M\times \tilde{M}$.

Let $\Omega$ be the volume form of $M$ extended to $M\times\tilde{M}$ naturally (more precisely, let $\Omega$ be the pullback of the volume form of $M$ under the projection $\pi_1:M\times\tilde{M}\rightarrow M$). Denote by $*\Omega$ the Hodge star of the restriction of $\Omega$ to $\Sigma_t$.  At any point $q\in \Sigma_t$, $*\Omega(q)=\Omega(e_1,\ldots,e_{2n})$ for any oriented orthonormal basis of $T_q\Sigma$. $*\Omega$ is the Jacobian of the projection $\pi_1$ from $\Sigma_t$ onto $M$. We shall show that $*\Omega$ remains positive along the mean curvature flow. By the implicit function theorem, this implies that $\Sigma_t$ is a graph over $M$.

We apply the result in the previous section to choose a basis that simplifies the evolution equation of $*\Omega$. Suppose $q\in \Sigma_t$ is of the form $q=(p, f(p))$ for $p\in M$ and $f(p)\in \tilde{M}$, and let $(a_1,\ldots,a_{2n})$ be the basis of $T_p M$ satisfying the properties listed in Proposition~\ref{basis}, for $L=Df_p:T_p M\rightarrow T_{f(p)}\tilde{M}$, with the inner products understood to be the metrics $g$ on $M$ at $p$ and $\tilde{g}$ on $\tilde{M}$ at $f(p)$. Thus we have \begin{equation}\label{domain_basis} a_1, a_2=Ja_1, \cdots, a_{2n-1}, a_{2n}=Ja_{2n-1}\end{equation} on $T_pM$. Define $E:T_pM\rightarrow T_{f(p)}\tilde{M}$ to be the isometry $E=Df_p[Df_p^*Df_p]^{-\frac{1}{2}}$ for $p\in M$. Let us also choose a basis of $T_{f(p)}\tilde{M}$ to be $(\tilde{a}_1,\ldots,\tilde{a}_{2n})=(E(a_1),\ldots,E(a_{2n}))$, as per Corollary~\ref{imagebasis}.

Then \begin{equation}\label{e_i}e_i=\frac{1}{\sqrt{1+|Df_p(a_i)|^2}}(a_i, Df_p(a_i))=\frac{1}{\sqrt{1+\lambda_i^2}}(a_i, \lambda_i E(a_i))\end{equation}
and
\begin{equation}\label{e_2n+i} e_{2n+i}=\mathcal{J}_{(p,f(p))}e_i\\
=\frac{1}{\sqrt{1+\lambda_i^2}}(J_p a_i, -\tilde{J}_{f(p)}\lambda_i E(a_i))=\frac{1}{\sqrt{1+\lambda_i^2}}(J_p a_i, -\lambda_i E(J_pa_i))\end{equation}
for $i=1,\ldots,2n$ form an orthonormal basis of $T_q (M\times \tilde{M})$. By construction, $e_1,\ldots,e_{2n}$ span $T_q\Sigma$, and $e_{2n+1},\ldots,e_{4n}$ span $N_q\Sigma$. In terms of this basis at each point $q\in\Sigma_t$:
\[*\Omega=\Omega(e_1,\ldots,e_{2n})=\frac{1}{\sqrt{\displaystyle\prod_{j=1}^{2n} (1+\lambda_j^2)}}.\]
The second fundamental form of $\Sigma_t$ is, at each point $q\in\Sigma_t$, characterized by coefficients
\begin{equation}\label{second_fund}h_{ijk}=G(\nabla^{M\times\tilde{M}}_{e_i} e_j, \mathcal{J}e_k).\end{equation}
Note that $h_{ijk}$ are completely symmetric with respect to $i,j,$ and $k$.

\vspace{5mm}
Before we prove Theorem~\ref{theorem}, we remark that the long time existence of the flow can be proved under more relaxed ambient curvature conditions, but the convergence of the flow does require the more refined properties of
the curvature of $\cp^n$.

\section{Proof of Theorem 1}
\subsection{Evolution of $*\Omega$ along the mean curvature flow}

In the rest of the paper we prove Theorem~\ref{theorem}. We use the following convention for indexes: for any index $i$ between $1$ and $2n$ , $i'$ denote the index $i+(-1)^{i+1}$. For example, $1'=2$ and $2'=1$. Unless otherwise is mentioned, all summation indexes range from $1$ to $2n$.

\begin{pro} \label{evolution}
Let $\Sigma$ be the graph of a symplectomorphism $f:(M,\omega)\rightarrow(\tilde{M},\tilde{\omega})$ between K\"ahler-Einstein manifolds $(M,g, \omega)$ and $(\tilde{M},\tilde{g}, \tilde{\omega})$ of real dimension $2n$ and of the same scalar curvature. Suppose the mean curvature flow $\Sigma_t$ with $\Sigma_0=\Sigma$ exists smoothly on $[0, t+\epsilon)$ for some $\epsilon>0$ and each $\Sigma_t$ is the graph of a symplectomorphism $f_t:(M, \omega)\rightarrow (\tilde{M}, \tilde{\omega})$.
At each point $q=(p, f_t(p))\in\Sigma_t$, $*\Omega$ satisfies the following equation:
\begin{align*}
\frac{d}{dt}*\Omega=&\Delta *\Omega+*\Omega\left[ Q(\lambda_i, h_{ijk}) +\displaystyle\sum_{i, k}\frac{\lambda_i^2}{(1+\lambda_k^2)(1+\lambda_i^2)}(R_{ikik}-\lambda_k^2\tilde{R}_{ikik})\right],
\end{align*}where
\begin{equation}\label{A(lambda_i)}\begin{split}Q(\lambda_i, h_{ijk})&=\displaystyle\sum_{i,j,k}h_{ijk}^2-2\displaystyle\sum_{k}\displaystyle\sum_{i<j}(-1)^{i+j}\lambda_i\lambda_j(h_{i'ik}h_{j'jk}-h_{i'jk}h_{j'ik}),\\\end{split}\end{equation}
$R_{ijkl}=R(a_i,a_j,a_k,a_l)$ and $\tilde{R}_{ijkl}=\tilde{R}(E(a_i),E(a_j),E(a_k),E(a_l))$ are, respectively, the coefficients of the curvature tensors $R$ and $\tilde{R}$ of $M$ and $\tilde{M}$ with respect to the chosen bases of $T_pM$ and $T_{f_t (p)}\tilde{M}$ that diagonalize $(Df_t)_p:T_pM\rightarrow T_{f_t(p)}\tilde{M}$, as per Proposition~\ref{basis} and Corollary~\ref{imagebasis}.

\end{pro}

\begin{proof}
The evolution equation of $*\Omega$ under mean curvature flow is, by Proposition 3.1 of \cite{wa3}:
\begin{align*}
\frac{d}{dt}*\Omega&=\Delta*\Omega+*\Omega(\displaystyle\sum_{i,j,k}h_{ijk}^2)-2\displaystyle\sum_{p,q,k}\displaystyle
\sum_{i<j}\Omega(e_1, \ldots, \underset{(i)}{\mathcal{J}e_p}, \ldots, \underset{(j)}{\mathcal{J}e_q},\ldots, e_{2n})h_{pik}h_{qjk}\\
&-\displaystyle\sum_{p,k,i}\Omega(e_1,\ldots,\underset{(i)}{\mathcal{J}e_p},\ldots, e_{2n})\mathcal{R}(\mathcal{J}e_p,e_k,e_k,e_i),\\
\end{align*} where $\mathcal{R}$ is the curvature tensor of $M\times\tilde{M}$.

We recall that $\Omega$ is a $2n$ form. The notation $\Omega(e_1, \ldots, \underset{(i)}{\mathcal{J}e_p}, \ldots, \underset{(j)}{\mathcal{J}e_q},\ldots, e_{2n})$ means that we replace $\mathcal{J}e_p$ in the $i$-th position and ${\mathcal{J}e_q}$ in the $j$-th position and similarly in the rest of the paper.

We denote
\[\mathcal{A}=*\Omega(\displaystyle\sum_{i,j,k}h_{ijk}^2)-2\displaystyle\sum_{p,q,k}\displaystyle\sum_{i<j}\Omega(e_1, \ldots, \underset{(i)}{\mathcal{J}e_p}, \ldots, \underset{(j)}{\mathcal{J}e_q},\ldots, e_{2n})h_{pik}h_{qjk}\] and
\begin{equation}\label{exp_B}\mathcal{B}=-\displaystyle\sum_{p,k,i}\Omega(e_1,\ldots,\underset{(i)}{\mathcal{J}e_p},\ldots, e_{2n})\mathcal{R}(\mathcal{J}e_p,e_k,e_k,e_i).\end{equation}
Since $\Omega$ only picks up the $\pi_1$ projection part, and
\begin{equation}\label{3.3}\pi_1( \mathcal{J} e_p)=\frac{1}{\sqrt{1+\lambda_p^2}}J a_p\end{equation} by (\ref{e_i}), $\mathcal{A}$ is equal to:
\begin{align*}
&*\Omega(\displaystyle\sum_{i,j,k}h_{ijk}^2)-2(*\Omega)\displaystyle\sum_{p,q,k}\displaystyle\sum_{i<j}
\frac{\sqrt{(1+\lambda_i^2)(1+\lambda_j^2)}}{\sqrt{(1+\lambda_p^2)(1+\lambda_q^2)}}\Omega(a_1, \ldots, \underset{(i)}{Ja_p}, \ldots,\underset{(j)}{Ja_q},\ldots, a_{2n})h_{pik}h_{qjk}.\\
\end{align*}

Recall the formula $J a_p=(-1)^{p+1} a_{p'}$ from \eqref{basis}. Fixing $i<j$, the term \[\Omega(a_1, \ldots, \underset{(i)}{Ja_p}, \ldots,\underset{(j)}{Ja_q},\ldots, a_{2n})\] is equal to \[\begin{split}
&(-1)^{p+1}(-1)^{q+1}\Omega(a_1,\ldots,\underset{(i)}{a_{p'}},\ldots,\underset{(j)}{a_{q'}},\ldots, a_{2n})\\
&=(-1)^{i+j}(\delta_{p i'}\delta_{q j'}-\delta_{pj'}\delta_{qi'}),\end{split}\] as only those terms with $p=i'$ and $q=j'$ or $p=j'$ and $q=i'$ survive. On the other hand, we have
\[\frac{\sqrt{(1+\lambda_i^2)}}{\sqrt{(1+\lambda_{i'}^2)}}=\lambda_i.\]
Therefore,
\[\begin{split}&\displaystyle\sum_{p,q,k}\displaystyle\sum_{i<j}
\frac{\sqrt{(1+\lambda_i^2)(1+\lambda_j^2)}}{\sqrt{(1+\lambda_p^2)(1+\lambda_q^2)}}\Omega(a_1, \ldots, \underset{(i)}{Ja_p}, \ldots,\underset{(j)}{Ja_q},\ldots, a_{2n})h_{pik}h_{qjk}\\
&=\displaystyle\sum_{k}\displaystyle\sum_{i<j}\lambda_i\lambda_j(-1)^{i+j}(h_{i'ik}h_{j'jk}-h_{i'jk}h_{j'ik}),\end{split}\]
and this shows that $\mathcal{A}=(*\Omega) Q(\lambda_i, h_{ijk})$.

On the other hand, switching the last two arguments $e_k$ and $e_i$ in \eqref{exp_B}, using (\ref{3.3}) again, and applying the skew-symmetry of curvature tensor, we derive
\begin{align*}
\mathcal{B}=&\displaystyle\sum_{p,k,i}\Omega(e_1\ldots,\underset{(i)}{\mathcal{J}e_p},\ldots, e_{2n})  \mathcal{R}(\mathcal{J}e_p,e_k,e_i,e_k)\\
&=*\Omega\displaystyle\sum_{p,k,i}\frac{\sqrt{1+\lambda_i^2}}{\sqrt{1+\lambda_p^2}}\Omega(a_1\ldots,\underset{(i)}{Ja_p},\ldots, a_{2n})\mathcal{R}(\mathcal{J}e_p,e_k,e_i,e_k)\\
&=*\Omega\displaystyle\sum_{k}\displaystyle\sum_{i}(-1)^i\lambda_i \mathcal{R}(\mathcal{J}e_{i'},e_k,e_i,e_k),
\end{align*} where we use $J a_p=(-1)^{p+1} a_{p'}$ and $\frac{\sqrt{(1+\lambda_i^2)}}{\sqrt{(1+\lambda_{i'}^2)}}=\lambda_i$ in the last equality.

Denote by $R$ and $\tilde{R}$ the curvature tensors of $M$ and $\tilde{M}$, respectively. We compute by Lemma \ref{QJcomm}, \eqref{e_i}, and \eqref{e_2n+i},
\begin{align*}
&\mathcal{R}(\mathcal{J}e_{i'},e_k,e_i,e_k)\\
&=R(\pi_1(\mathcal{J}e_{i'}),\pi_1(e_k), \pi_1(e_i),\pi_1(e_k))+\tilde{R}(\pi_2(\mathcal{J}e_{i'}),\pi_2(e_k), \pi_2(e_i),\pi_2(e_k))\\
&=\frac{1}{(1+\lambda_k^2)\sqrt{(1+\lambda_i^2)(1+\lambda_{i'}^2)}}[R(Ja_{i'},a_k,a_i,a_k)-\lambda_k^2\lambda_i\lambda_{i'}\tilde{R}(\tilde{J}E(a_{i'}),E(a_k),E(a_i),E(a_k))]\\
&=\frac{1}{(1+\lambda_k^2)\sqrt{(1+\lambda_i^2)(1+\lambda_{i'}^2)}}[R(Ja_{i'},a_k,a_i,a_k)-\lambda_k^2\tilde{R}(E(Ja_{i'}),E(a_k),E(a_i),E(a_k))]\\
&=\frac{1}{(1+\lambda_k^2)\sqrt{(1+\lambda_i^2)(1+\lambda_{i'}^2)}}[(-1)^iR(a_i,a_k,a_i,a_k)-(-1)^i\lambda_k^2\tilde{R}(E(a_i),E(a_k),E(a_i),E(a_k))]\\
&=\frac{(-1)^i}{(1+\lambda_k^2)\sqrt{(1+\lambda_i^2)(1+\lambda_{i'}^2)}}(R_{ikik}-\lambda_k^2\tilde{R}_{ikik})\\
&=\frac{(-1)^i \lambda_i}{(1+\lambda_k^2)(1+\lambda_{i}^2)}(R_{ikik}-\lambda_k^2\tilde{R}_{ikik}).
\end{align*}
\end{proof}

The ambient curvature term $\mathcal{B}$ can be further simplified when $M=\tilde{M}=\cp^n$.
\begin{cor} \label{cpn} Under the same assumption as in Proposition \ref{evolution}, if in addition $M$ and $\tilde{M}$ are both $\cp^n$ with the Fubini-Study metric, then:
\begin{align*}
\frac{d}{dt}*\Omega=&\Delta *\Omega+*\Omega \left[Q (\lambda_i, h_{ijk})+\sum_{k\text{ odd}}\frac{(1-\lambda_k^2)^2}{(1+\lambda_k^2)^2}\right].
\end{align*}

\end{cor}

\begin{proof}
On $\cp^n$ with the Fubini-Study metric $\langle \cdot, \cdot\rangle$, the sectional curvature is (see for example \cite{ls}):
\[K(X,Y)=\frac{\frac{1}{4}(||X\wedge Y||^2+3\langle JX, Y \rangle^2)}{|X|^2|Y|^2-\langle X,Y\rangle^2}.\] Therefore, with respect to the chosen orthonormal bases of $T_xM$ and $T_{f(x)}\tilde{M}$, the sectional curvatures $K$ and $\tilde{K}$ of $M$ and $\tilde{M}$ are:
\[K(a_i,a_{i'})=1 \text{ and } K(a_r,a_s)=\frac{1}{4} \text{ for all other $r, s$}; \text{ and}\]
\[\tilde{K}(E(a_i),E(a_{i'}))=1 \text{ and } \tilde{K}(E(a_r),E(a_s))=\frac{1}{4} \text{ for all other $r, s$}.\]

Therefore,
\[R_{ikik}=K(a_i,a_k)=\frac{1}{4}(1+3\delta_{ik'})\] and
\[\tilde{R}_{ikik}=\tilde{K}(E(a_i),E(a_k))=\frac{1}{4}(1+3\delta_{ik'})\]
for any $i, k$ with $i\not=k$.

Plugging these into the expression for $\mathcal{B}$, we obtain
\begin{align*}
\mathcal{B}=&\frac{*\Omega}{4}\displaystyle\sum_{k}\displaystyle\sum_{i\neq k}\frac{\lambda_i^2(1-\lambda_k^2)}{(1+\lambda_k^2)(1+\lambda_{i}^2)}(1+3\delta_{ik'})\\
&=*\Omega\displaystyle\sum_{k}\frac{\lambda_{k'}(1-\lambda_k^2)}{(1+\lambda_k^2)(\lambda_k+\lambda_{k'})}+\frac{*\Omega}{4}\displaystyle\sum_{k}\frac{1-\lambda_k^2}{1+\lambda_k^2}\left(\displaystyle\sum_{i\neq k, k'}\frac{\lambda_i}{\lambda_i+\lambda_{i'}}\right)\\
\end{align*} by dividing it into two summands with $i=k'$ and $i=k'$.
Using $\lambda_k\lambda_{k'}=1$ and $\displaystyle\sum_{i\neq k, k'}\frac{\lambda_i}{\lambda_i+\lambda_{i'}}=\displaystyle\sum_{i \text{ odd } \neq k, k'}\frac{\lambda_i+\lambda_i'}{\lambda_i+\lambda_{i'}}=n-1$, we derive
\begin{align*}
\mathcal{B}=*\Omega\displaystyle\sum_{k}\frac{1-\lambda_k^2}{(1+\lambda_{k}^2)^2}+\frac{(n-1)}{4}*\Omega\displaystyle\sum_{k}\frac{1-\lambda_k^2}{1+\lambda_k^2}.
\end{align*}

The second term vanishes as sums with odd $k$ and even $k$ cancel with each other. Finally, we arrive at:
\begin{align*}
\mathcal{B}=*\Omega &\displaystyle\sum_{k}\frac{1-\lambda_{k}^2}{(1+\lambda_{k}^2)^2}=*\Omega\displaystyle\sum_{k\text{ odd}}\frac{(1-\lambda_k^2)^2}{(1+\lambda_k^2)^2}.
\end{align*}

\end{proof}

In this case $\mathcal{B}\geq 0$, with equality holding if and only if all the singular values of $f$ are equal (and thus necessarily equal to 1). Moreover, $\frac{(1-\lambda_k^2)^2}{(1+\lambda_k^2)^2}<1$, so $\mathcal{B}< n(*\Omega)\leq \frac{n}{2^n}$.

We notice that $Q(\lambda_i, h_{ijk})$ is a quadratic form in $h_{ijk}$ which can be rewritten as
\begin{equation}\begin{split} \label{quadform}
{Q}(\lambda_i, h_{ijk})=&\displaystyle\sum_{i,j,k}h_{ijk}^2-2\displaystyle\sum_{k}\displaystyle\sum_{i\text{ odd}}(h_{iik}h_{i'i'k}-h_{ii'k}^2)\\
&-2\displaystyle\sum_{k}\displaystyle\sum_{i \text{ odd}<j\text{ odd}}(\lambda_i-\lambda_{i'})(\lambda_j-\lambda_{j'})h_{i'ik}h_{j'jk}\\
&-2\displaystyle\sum_{k}\displaystyle\sum_{i\text{ odd}<j\text{ odd}}[-(\lambda_i\lambda_j+\lambda_{i'}\lambda_{j'})h_{i'jk}h_{j'ik}+(\lambda_{i'}\lambda_j+\lambda_i\lambda_{j'})h_{ijk}h_{j'i'k}].
\end{split}
\end{equation}

\begin{lem} \label{evalue}
When each $\lambda_i=1$, \[{Q}((1,\ldots,1), h_{ijk})\geq (3-\sqrt{5}) ||h_{ijk}||^2\] where \[||h_{ijk}||^2=\sum_{i} h_{iii}^2+\sum_{i\not= j} h_{ijj}^2+\sum_{i<j<k} h_{ijk}^2.\]
\end{lem}

\begin{proof}
See Appendix.

\end{proof}

\begin{pro} \label{pinchconst} Let $Q(\lambda_i, h_{jkl})$ be the quadratic form defined in Proposition~\ref{evolution}.
In each dimension $n$, there exist $\Lambda_0>1$ such that $Q(\lambda_i, h_{jkl})$ is non-negative whenever $\frac{1}{{\Lambda_0}}\leq\lambda_i\leq{\Lambda_0}$ for $i=1,\ldots,2n$. Moreover, for any $1\leq \Lambda_1<\Lambda_0$, there exists a $\delta>0$
such that \[Q(\lambda_i, h_{jkl})\geq \delta \displaystyle\sum_{i,j,k}h_{ijk}^2\] whenever $\frac{1}{{\Lambda_1}}\leq\lambda_i\leq {\Lambda_1}$ for $i=1,\ldots,2n$.

\end{pro}

\begin{proof} Since $\frac{1}{6}\sum_{i,j,k}h_{ijk}^2\leq ||h_{ijk}||^2\leq \sum_{i,j,k}h_{ijk}^2$,
by Lemma~\ref{evalue}, \[Q((1,\cdots, 1), h_{ijk}) \geq \frac{3-\sqrt{5}}{6}\sum_{i, j, k}h_{ijk}^2.\]
Since being a positive definite matrix is an open condition, there is an open neighborhood $U$ of $(\lambda_1,\ldots, \lambda_{2n})=(1,\cdots, 1)$ such that $ (\lambda_1,\ldots, \lambda_{2n})\in U$ implies $Q(\lambda_i, h_{ijk})$ is positive definite. Let $\delta_{\vec{\lambda}}$ be the smallest eigenvalue of $Q$ at $\vec{\lambda}\equiv(\lambda_1,\ldots,\lambda_{2n})$. Note that $\delta_{\vec{\lambda}}$ is a continuous function in $\vec{\lambda}$ and set
\[\delta_\Lambda=\displaystyle\min \{ \delta_{\vec{\lambda}}\,|\, \vec{\lambda}=(\lambda_1,\ldots,\lambda_{2n}) \text{ and }\frac{1}{{\Lambda}}\leq\lambda_i\leq{\Lambda} \text{ for } i=1,\ldots,2n\}. \] $\Lambda_0$ defined by
  \[\Lambda_0\equiv\sup \{ \Lambda \,|\, \Lambda\geq 1 \text{ and } \delta_\Lambda > 0\}\] has the desired property.

\end{proof}

\begin{rem}
$\Lambda_0$ is computable in each dimension $n$. In particular, $\Lambda_0=\infty$ when $n=1$, and $\Lambda_0=\frac{2}{5}\sqrt{10}+\frac{1}{5}\sqrt{15}$ when $n=2$. This can be checked by dividing $Q$ into smaller quadratic forms and compute the eigenvalues as in the Appendix.
\end{rem}

\begin{cor}\label{pccpn}
 Under the same assumption as in Proposition \ref{evolution}, suppose in addition that $M$ and $\tilde{M}$ are both $\cp^n$ with the Fubini-Study metric. There exist constants $\Lambda_0>1$, depending only on $n$, such that for any $\Lambda_1$, $1\leq \Lambda_1<\Lambda_0$ there exists a $\delta>0$ with
\begin{equation} \label{pdineq}
\left(\frac{d}{dt}-\Delta \right)*\Omega \geq \delta*\Omega |{\textrm{II}}|^2 +*\Omega \sum_{k\text{ odd}}\frac{(1-\lambda_k^2)^2}{(1+\lambda_k^2)^2},
\end{equation} whenever $\frac{1}{{\Lambda_1}}\leq\lambda_i\leq{\Lambda_1}$ for every $i$. Here $|\textrm{II}|$ is the norm of the second fundamental form of $\Sigma_t$.
\end{cor}

We recall the norm of the second fundamental form is
 \[\begin{split}|\textrm{II}|&=\sqrt{\displaystyle\sum_{i,j, k,l}G^{ik} G^{jl}G( \textrm{II}(w_i,w_j),\textrm{II}(w_k,w_l))}\\
 &=\sqrt{\displaystyle\sum_{i,j,k,l,r,s}G^{ik} G^{jl} G^{rs}G (\nabla^{M\times\tilde{M}}_{w_i}w_j, \mathcal{J} w_r ) G( \nabla^{M\times\tilde{M}}_{w_k}w_l, \mathcal{J} w_s )}\end{split}\] with respect to an arbitrary basis $w_1,\ldots,w_{2n}$ of $T_q\Sigma$ with $G_{ij}=G( w_i, w_j)$ and $G^{ij}=(G_{ij})^{-1}$. By (\ref{second_fund}), \[|\textrm{II}|=\sqrt{\displaystyle\sum_{i,j,k}h_{ijk}^2}\] for the chosen basis~\eqref{e_i}.

\begin{proof}
The result follows from Corollary~\ref{cpn} and Proposition~\ref{pinchconst}.
\end{proof}

\subsection{Preservation of graphical and pinching conditions}

Short-time existence of the mean curvature flow in question is guaranteed by general theory of quasilinear parabolic PDE. In order to establish long-time existence and convergence, we shall show that when an appropriate pinching holds initially, then $f$ remains $\Lambda_0$-pinched along the flow, $*\Omega$ satisfies the differential inequality~\eqref{pdineq} along the flow, and $\displaystyle\min_{\Sigma_t} *\Omega$ is non-decreasing in time. First we make several preliminary observations. We consider $\frac{1}{\sqrt{\displaystyle\prod_{i}(1+\lambda_i^2)}}$, for $\lambda_i>0$, $\lambda_i\lambda_{i'}=1$, where $i'=i+(-1)^{i+1}$, $i=1,\ldots,2n$ (in other words, $\lambda_{2k-1}\lambda_{2k}=1$ for $k=1,\ldots,n$). It can be rewritten as:
\[\frac{1}{\sqrt{\displaystyle\prod_{i}(1+\lambda_i^2)}}=\frac{1}{\displaystyle\prod_{i\,\, \text{odd}}(\lambda_i+\lambda_{i'})}.\]
This expression always has an upper bound: $\lambda_i\lambda_{i'}=1$ implies that $\lambda_i+\lambda_{i'}\geq 2$. Therefore,
\begin{equation}\label{upperbound}
\frac{1}{\sqrt{\displaystyle\prod_{i}(1+\lambda_i^2)}}\leq \frac{1}{2^n},
\end{equation}
with equality if and only if $\lambda_i=1$ for all $i$.

If $\lambda_i$'s are bounded, $\frac{1}{\sqrt{\displaystyle\prod_{i}(1+\lambda_i^2)}}$ also has a positive lower bound.

\begin{lem} \label{pinchbound}
If $\frac{1}{{\Lambda}}\leq \lambda_i\leq {\Lambda}$ for all $i$, where $\Lambda>1$, then:
\[\frac{1}{2^n}-\epsilon\leq\frac{1}{\sqrt{\displaystyle\prod_{i}(1+\lambda_i^2)}},\]
where $\epsilon=\frac{1}{2^n}-\frac{1}{({\Lambda}+\frac{1}{{\Lambda}})^n}>0$.
\end{lem}

\begin{proof}
The function $x+\frac{1}{x}$ is increasing when $x>1$. Therefore if $\frac{1}{{\Lambda}}\leq \lambda_i\leq {\Lambda}$ for all $i$, then
\[\lambda_i+\lambda_{i'}\leq {\Lambda}+\frac{1}{{\Lambda}}.\]
It follows that
\[\frac{1}{2^n}-\epsilon\leq\frac{1}{\sqrt{\displaystyle\prod_{i}(1+\lambda_i^2)}}\leq\frac{1}{2^n},\]
where $\epsilon=\frac{1}{2^n}-\frac{1}{({\Lambda}+\frac{1}{{\Lambda}})^n}$.

\end{proof}

On the other hand, a positive lower bound on $\frac{1}{\sqrt{\displaystyle\prod_{i}(1+\lambda_i^2)}}$ implies a bound on each $\lambda_i$.

\begin{lem} \label{boundpinch}
If $\frac{1}{2^n}-\epsilon\leq\frac{1}{\sqrt{\displaystyle\prod_{i}(1+\lambda_i^2)}}$, where $0<\epsilon<\frac{1}{2^n}$, then:
\[\frac{1}{{\Lambda}}\leq \lambda_i\leq {\Lambda}\]
for all $i=1,\ldots,2n$, where $\Lambda=\frac{\frac{1}{2^n}}{\frac{1}{2^n}-\epsilon}+\sqrt{\left(\frac{\frac{1}{2^n}}{\frac{1}{2^n}-\epsilon}\right)^2-1}>1$.
\end{lem}

\begin{proof}
If
\[\frac{1}{2^n}-\epsilon\leq\frac{1}{\sqrt{\displaystyle\prod_{i}(1+\lambda_i^2)}}=\frac{1}{\displaystyle\prod_{i \text{ odd}}(\lambda_i+\lambda_{i'})},\]
then
\[\displaystyle\prod_{i \text{ odd}}(\lambda_i+\lambda_{i'})\leq \frac{2^n}{1-2^n\epsilon}\] and
\[\lambda_i+\lambda_{i'}\leq \frac{2^n}{(1-2^n\epsilon)\displaystyle\prod_{j\neq i, j \text{ odd}}(\lambda_j+\lambda_{j'})}\]
for each $i$.

Since $\lambda_j+\lambda_{j'}\geq 2$ for each $j$, the inequality implies
\[\lambda_i+\lambda_{i'}\leq 2\frac{\frac{1}{2^n}}{\frac{1}{2^n}-\epsilon}\]
Since $\lambda_i\lambda_{i'}=1$, it follows that:
\[\frac{1}{{\Lambda}} \leq\lambda_i\leq {\Lambda}\]
where $\Lambda=\frac{\frac{1}{2^n}}{\frac{1}{2^n}-\epsilon}+\sqrt{\left(\frac{\frac{1}{2^n}}{\frac{1}{2^n}-\epsilon}\right)^2-1}$.

\end{proof}

After these algebraic preliminaries, we return to the mean curvature flow. Recall that $f$ is $\Lambda$-pinched in the sense of Definition \ref{pinching}, if $\frac{1}{\Lambda}\leq \lambda_i\leq \Lambda$ at each point $p\in M$ in which $\lambda_i$'s are the singular values of $Df_p$ as in section 2.2.

\begin{pro}\label{lbpreserved} 
Let $\Sigma_t$ be the mean curvature flow of the graph $\Sigma$ of a symplectomorphism $f:M\rightarrow \tilde{M}$ where $M=\tilde{M}=\cp^n$ with the Fubini-Study metric. Suppose $\Sigma_t$ exists smoothly on $[0, T)$ for some $T>0$. Let $*\Omega$ be the Jacobian of the projection $\pi_1:\Sigma_t\rightarrow M$.
Let $\Lambda_0$ be the constants characterized by Proposition~\ref{pinchconst}.

If $*\Omega$ has the initial lower bound:
\[\frac{1}{2^n}-\epsilon\leq *\Omega \]
for $\epsilon=\frac{1}{2^n}\left(1-\frac{2}{{\Lambda'}+\frac{1}{{\Lambda'}}}\right)$ for some $1<\Lambda'<\Lambda_0$, then $\displaystyle\min_{\Sigma_t}*\Omega$ is nondecreasing as a function in $t$. In particular, $\Sigma_t$ is the graph of a symplectomorphism $f_t:M\rightarrow \tilde{M}$.
\end{pro}

\begin{proof}
 If initially $\frac{1}{2^n}-\epsilon\leq *\Omega$ for $\epsilon=\frac{1}{2^n}\left(1-\frac{2}{{\Lambda'}+\frac{1}{{\Lambda'}}}\right)$. We compute that
 $\frac{\frac{1}{2^n}}{\frac{1}{2^n}-\epsilon}=\frac{\Lambda'+\frac{1}{\Lambda'}}{2}$. Thus, by Lemma~\ref{boundpinch}, $f$ is $\Lambda'$-pinched. That in turn implies that $*\Omega$ initially satisfies inequality~\eqref{pdineq}, and in particular,
\begin{equation}\label{nonneg}
\left(\frac{d}{dt}-\Delta \right)*\Omega \geq *\Omega \sum_{k \,\,odd}\frac{(1-\lambda_k^2)^2}{(1+\lambda_k^2)}.
\end{equation}
Thus $*\Omega>\frac{1}{2^n}-\epsilon$ for some $[0, T')$ with $T'<T$.

Suppose at $T'$, $*\Omega=\frac{1}{2^n}-\epsilon$ for the first time after $t=0$. But in $[0, T')$, we have $*\Omega>\frac{1}{2^n}-\epsilon$ and thus $f$ is $\Lambda'$-pinched and inequality \eqref{nonneg} is satisfied again. Since the right hand side of \eqref{nonneg} is strictly positive unless $*\Omega=\frac{1}{2^n}$, $\min_{\Sigma_t} *\Omega$ is non-decreasing in time by the maximum principle.

\end{proof}

\begin{cor}\label{cor} Under the same assumption as in Proposition \ref{lbpreserved},
if the initial symplectomorphism $f$ is $\Lambda_1$-pinched, for \[\Lambda_1=\left[\frac{1}{2}\left({\Lambda_0}+\frac{1}{{\Lambda_0}}\right)\right]^{\frac{1}{n}}+\sqrt{\left[\frac{1}{2}\left({\Lambda_0}+\frac{1}{{\Lambda_0}}\right)\right]^{\frac{2}{n}}-1}<\Lambda_0,\] then each $f_t$ is  $\Lambda_0$-pinched along the mean curvature flow.
\end{cor}

\begin{proof} The proof consists of only algebraic manipulation and there is no need to apply the maximum principle again. We need a simple algebraic formula which can be easily verified: for $x>1,y>1$, \begin{equation}\label{identity}x+\sqrt{x^2-1}=y\,\,\text{if and only if} \,\,x=\frac{y+y^{-1}}{2}.\end{equation}

By the definition of $\Lambda_1$,
\begin{equation}\label{0_1_comp}\frac{1}{2}\left({\Lambda_1}+\frac{1}{{\Lambda_1}}\right)=\left(\frac{1}{2}\left({\Lambda_0}
+\frac{1}{{\Lambda_0}}\right)\right)^\frac{1}{n}\end{equation}
which is less than $\frac{1}{2}\left({\Lambda_0}+\frac{1}{{\Lambda_0}}\right)$ because ${\Lambda_0}+\frac{1}{{\Lambda_0}}>2$.
Since $\Lambda_0>1$ and $\Lambda_1>1$, it follows that $\Lambda_1<\Lambda_0$.

Now suppose $f$ is initially $\Lambda_1$-pinched, by Lemma~\ref{pinchbound}, $*\Omega$ has initial lower bound:
\[\frac{1}{2^n}-\epsilon\leq *\Omega \]
for \begin{equation}\label{ep}\epsilon=\frac{1}{2^n}-\frac{1}{({\Lambda_1}+\frac{1}{{\Lambda_1}})^n}.\end{equation} Then, by Proposition~\ref{lbpreserved}, the lower bound of $*\Omega$ remains true along the flow. Lemma~\ref{boundpinch} then implies that $f$ is $\Lambda'$-pinched along the flow for
\begin{equation}\label{lambda_p}\Lambda'=\frac{\frac{1}{2^n}}{\frac{1}{2^n}-\epsilon}+\sqrt{\left(\frac{\frac{1}{2^n}}
{\frac{1}{2^n}-\epsilon}\right)^2-1}.\end{equation} We claim that with the given $\Lambda_1$ and  $\epsilon$ given by \eqref{ep}, $\Lambda'$ is exactly $\Lambda_0$. In fact, from \eqref{lambda_p} and \eqref{identity}, we obtain

\[\frac{\frac{1}{2^n}}{\frac{1}{2^n}-\epsilon}=\frac{1}{2}\left({\Lambda'}+\frac{1}{\Lambda'}\right).\] On the other hand from \eqref{ep}, we solve $\frac{\frac{1}{2^n}}{\frac{1}{2^n}-\epsilon}=\left(\frac{1}{2}\left({\Lambda_1}
+\frac{1}{{\Lambda_1}}\right)\right)^{n}=\frac{1}{2}\left({\Lambda_0}
+\frac{1}{{\Lambda_0}}\right)$ by \eqref{0_1_comp}. Therefore $f$ is $\Lambda_0$ pinched along the flow.

\end{proof}

We believed that the constant $\Lambda_1$ can be further improved by considering the evolution equation of $\lambda_i$ directly. In this article, we find that the evolution equation of $*\Omega$ is sufficient to yield the desired constant, albeit not an optimal one.

In Theorem 1, we choose a $\Lambda$ that is slightly less than $\Lambda_1$ in Corollary \ref{cor}, then $f_t$ will $\Lambda_0'$ pinched along the flow for some $\Lambda_0'< \Lambda_0$ and thus by Corollary \ref{pccpn}, we have \eqref{pdineq} all the way along the flow. We shall see that this is enough for the long time existence and convergence.

\subsection{Long-time existence of the mean curvature flow}

We assume $M=\tilde{M}=\mathbb{CP}^n$. To prove long-time existence of the flow, we follow the method in \cite{wa3}. We isometrically embed $M\times\tilde{M}$ into $\R^N$. The mean curvature flow equation in terms of the coordinate function $F(x,t)\in \R^N$ is:
\[\frac{d}{dt}F(x,t)=H=\bar{H}+V,\]
where $H\in T(M\times\tilde{M})/T\Sigma_t$ is the mean curvature vector of $\Sigma_t$ in $M$, $\bar{H}\in T\R^N/T\Sigma_t$  is the mean curvature vector of $\Sigma_t$ in $\R^N$, and $V=-\sum_{a} \textrm{II}_{M\times\tilde{M}}(e_a,e_a)$ where $\{e_a\}_{a=1\cdots 2n} $ is an orthonormal basis of $T\Sigma_t$.
In the following calculation, the index $a$ is summed from $1$ to $2n$, \begin{align*}
H=\pi^{M\times\tilde{M}}_{N\Sigma}(\nabla^{M\times\tilde{M}}_{e_a} e_a)&=\nabla^{M\times\tilde{M}}_{e_a} e_a-\nabla^\Sigma_{e_a} e_a\\
&=\nabla^{\R^N}_{e_a} e_a-\pi^{\R^N}_{N(M\times\tilde{M})}(\nabla^{\R^N}_{e_a} e_a)-\nabla^\Sigma_{e_a} e_a\\
&=\nabla^{\R^N}_{e_a} e_a-\nabla^\Sigma_{e_a} e_a+V\\
&=\pi^{\R^N}_{N\Sigma}(\nabla^\Sigma_{e_a} e_a)+V\\
&=\bar{H}+V.
\end{align*}
Note that $V$ is bounded since both $M$ and $\tilde{M}$ are compact.

\vspace{5mm}
Following \cite{wa3}, we assume that there is a singularity at space time point $(y_0,t_0)\in \R^N \times \R$. Consider the backward heat kernel of Huisken $\rho_{y_0,t_0}$ at $(y_0,t_0)$:
\[\rho_{y_0,t_0}(y,t)=\frac{1}{4\pi(t_0-t)^n}\exp \left(\frac{-|y-y_0|^2}{4(t_0-t)}\right).\]

Let $d\mu_t$ denote the volume form of $\Sigma_t$. By Huisken's monotonicity formula \cite{hu2},  $\displaystyle\lim_{t\rightarrow t_0}\int \rho_{y_0,t_0}d\mu_t$ exists.

\begin{lem}\label{limitandineq}
The limit $\displaystyle\lim_{t\rightarrow t_0}\int (1-*\Omega)\rho_{y_0,t_0}d\mu_t$ exists and:
\[
\frac{d}{dt}\int (1-*\Omega)\rho_{y_0,t_0}d\mu_t\leq C-\delta\int*\Omega|\textrm{II}|^2\rho_{y_0,t_0}d\mu_t
\]
for some constant $C>0$.
\end{lem}

\begin{proof}
By \cite{wa8}:
\[\frac{d}{dt}\rho_{y_0,t_0}=-\Delta\rho_{y_0,t_0}-\rho_{y_0,t_0}\left(\frac{|F^\bot|^2}{4(t_0-t)^2}+\frac{F^\bot\cdot \bar{H}}{t_0-t}+\frac{F^\bot\cdot V}{2(t_0-t)}\right)\]
where $F^\bot\in T\R^N/T\Sigma_t$ is the orthogonal component of $F\in T\R^N$.

By \cite{wa3}:
\[\frac{d}{dt}d\mu_t=-|H|^2d\mu_t=-\bar{H}\cdot(\bar{H}+V)d\mu_t.\]

Combining these results, we obtain:
\begin{align*}
&\frac{d}{dt}\int (1-*\Omega)\rho_{y_0,t_0}d\mu_t\\
&\leq  \int [\Delta(1-*\Omega)-\delta*\Omega|\textrm{II}|^2]\rho_{y_0,t_0}d\mu_t\\
&-\int (1-*\Omega)\left[\Delta\rho_{y_0,t_0}+\rho_{y_0,t_0}\left(\frac{|F^\bot|^2}{4(t_0-t)^2}+\frac{F^\bot\cdot \bar{H}}{t_0-t}+\frac{F^\bot\cdot V}{2(t_0-t)}\right)\right]\\
&-\int(1-*\Omega)[\bar{H}\cdot(\bar{H}+V)]\rho_{y_0,t_0}d\mu_t\\
&=\int [\Delta(1-*\Omega)\rho_{y_0,t_0}-(1-*\Omega)\Delta\rho_{y_0,t_0}]d\mu_t-\delta\int*\Omega|\textrm{II}|^2\rho_{y_0,t_0}d\mu_t\\
&-\int(1-*\Omega)\rho_{y_0,t_0}\left[\left(\frac{|F^\bot|^2}{4(t_0-t)^2}+\frac{F^\bot\cdot \bar{H}}{t_0-t}+\frac{F^\bot\cdot V}{2(t_0-t)}\right)+|\bar{H}|^2+\bar{H}\cdot V\right]d\mu_t\\
&=-\delta\int*\Omega|\textrm{II}|^2\rho_{y_0,t_0}d\mu_t-\int(1-*\Omega)\rho_{y_0,t_0}\left|\frac{F^\bot}{2(t_0-t)}+\bar{H}+\frac{V}{2}\right|^2d\mu_t\\
&+\int(1-*\Omega)\rho_{y_0,t_0}\left|\frac{V}{2}\right|^2d\mu_t.
\end{align*}

Since $V$ is bounded, and since $\int(1-*\Omega)\rho_{y_0,t_0}d\mu_t\leq \int \rho_{(y_0,t_0)}d\mu_t<\infty$, it follows that:
\[
\frac{d}{dt}\int (1-*\Omega)\rho_{y_0,t_0}d\mu_t\leq C-\delta\int*\Omega|\textrm{II}|^2\rho_{y_0,t_0}d\mu_t
\]
for some constant $C$. Now $F(t)=\int (1-*\Omega)\rho_{y_0,t_0}d\mu_t$ is non-negative and $F'(t)\leq C$, or $F(t)-Ct$ is non-increasing in $t\in [0, t_0)$. From this it follows that the limit as $t\rightarrow t_0$ exists.

\end{proof}

For $\nu>1$, the parabolic dilation $D_\nu$ at $(y_0,t_0)$ is defined by:
\[D_\nu:\R^N\times[0,t_0)\rightarrow \R^N \times[-\nu^2t_0,0),\]
\[(y,t)\mapsto (\nu(y-y_0),\nu^2(t-t_0)).\]

Let $\mathcal{S}\subset \R^N\times[0,t_0)$ be the total space of the mean curvature flow, and let $\mathcal{S}_\nu\equiv D_\nu(\mathcal{S})\subset \R^N\times[-\nu^2t_0,0)$. If $s$ denotes the new time parameter, then $t=t_0+\frac{s}{\nu^2}$.

Let $d\mu^\nu_s$ be the induced volume form on $\Sigma$ by $F^\nu_s\equiv \nu F_{t_0+\frac{s}{\nu^2}}$. The image of $F^\nu_s$ is the $s-$slice of $\mathcal{S}_\nu$, denoted $\Sigma^\nu_s$.
 \begin{rem}\label{scaleinv}
Note that:
\[\int(1-*\Omega)\rho_{y_0,t_0}d\mu_t=\int(1-*\Omega)\rho_{0,0}d\mu^\nu_s\]
because $*\Omega$ and $\rho_{y_0,t_0}d\mu_t$ are invariant under parabolic dilation.
\end{rem}

\begin{lem}
For any $\tau>0$:
\[\displaystyle\lim_{\nu\rightarrow \infty}\int_{-1-\tau}^{-1}\int*\Omega|\textrm{II}|^2\rho_{0,0}d\mu^\nu_s ds=0.\]

\end{lem}

\begin{proof}
From Remark~\ref{scaleinv}:
\begin{align*}
\frac{d}{ds}\int(1-*\Omega)\rho_{0,0}d\mu^\nu_s&=\frac{1}{\nu^2}\frac{d}{dt}\int (1-*\Omega)\rho_{y_0,t_0}d\mu_t.
\end{align*}

Then by Lemma~\ref{limitandineq}:
\begin{align*}
\frac{d}{ds}\int(1-*\Omega)\rho_{0,0}d\mu^\nu_s\leq \frac{C}{\nu^2}-\frac{\delta}{\nu^2}\int*\Omega|\textrm{II}|^2\rho_{y_0,t_0}d\mu_t
\end{align*}
for some constant $C$.
But $\frac{1}{\nu^2}\int*\Omega|\textrm{II}|^2\rho_{y_0,t_0}d\mu_t=\int*\Omega|\textrm{II}|^2\rho_{0,0}d\mu^\nu_s$ since the norm of the second fundamental form scales like the inverse of the distance, so:
\[\frac{d}{ds}\int(1-*\Omega)\rho_{0,0}d\mu^\nu_s\leq \frac{C}{\nu^2}-\delta\int*\Omega|\textrm{II}|^2\rho_{0,0}d\mu^\nu_s.\]

Integrating this inequality with respect to $s$ from $-1-\tau$ to $-1$, we obtain:
\[ \delta\int_{-1-\tau}^{-1}\int*\Omega|\textrm{II}|^2\rho_{0,0}d\mu^\nu_s ds\leq -\int(1-*\Omega)\rho_{0,0}d\mu^\nu_{-1}+\int(1-*\Omega)\rho_{0,0}d\mu^\nu_{-1-\tau}+\frac{C}{\nu^2}.\]

By Remark~\ref{scaleinv} and the fact that $\displaystyle\lim_{t\rightarrow t_0}\int(1-*\Omega)\rho_{y_0,t_0}d\mu_t$ exists (Lemma~\ref{limitandineq}), the right-hand side of the inequality above approaches zero as $\nu\rightarrow \infty$.

\end{proof}

We take a sequence $\nu_j\rightarrow\infty$. Then for a fixed $\tau$:
\[\int_{-1-\tau}^{-1}\int*\Omega|\textrm{II}|^2\rho_{0,0}d\mu^{\nu_j}_s ds \leq C(j)\]
where $C(j)\rightarrow 0$.

Choose $\tau_j\rightarrow 0$ such that $\frac{C(j)}{\tau_j}\rightarrow 0$, and $s_j\in[-1-\tau_j,-1]$ so that
\begin{equation}\label{zerolimit}
\int*\Omega|\textrm{II}|^2\rho_{0,0}d\mu^{\nu_j}_{s_j} \leq \frac{C(j)}{\tau_j}.
\end{equation} Observe that
\[\rho_{0,0}(F^{\nu_j}_{s_j}, s_j)=\frac{1}{(4\pi(-s_j)^2)^n}\exp\left(\frac{-|F^{\nu_j}_{s_j}|^2}{4(-s_j)}\right).
\]

When $j$ is large enough, we may assume that $\tau_j\leq 1$, and thus that $s_j\in[-2,-1]$. For a ball centered at 0 of radius $R>0$, $B_R(0)\in\R^N$, we have:
\[ \int*\Omega|\textrm{II}|^2\rho_{0,0}d\mu^{\nu_j}_{s_j} \geq C'\int_{\Sigma^{\nu_j}_{s_j}\cap B_R(0)}*\Omega|\textrm{II}|^2d\mu^{\nu_j}_{s_j}
\] for a constant $C'>0$, since $s_j$ are bounded and since $|F^{\nu_j}_{s_j}|\leq R$ on $\Sigma^{\nu_j}_{s_j}\cap B_R(0)$.

Then by inequality~\eqref{zerolimit} and the fact that $*\Omega$ has a positive lower bound, we conclude the following result.

\begin{lem}\label{curvintzero}
For any compact set $\mathcal{K}\subset\R^N$:
\[\int_{\Sigma^{\nu_j}_{s_j}\cap \mathcal{K}}|\textrm{II}|^2d\mu^{\nu_j}_{s_j}\rightarrow 0\]
as $j\rightarrow \infty$.
\end{lem}

Then, as shown in \cite{wa3}, it follows that
\[
\displaystyle\lim_{t\rightarrow t_0}\int\rho_{y_0,t_0}d\mu_t\leq 1.
\] Finally, White's theorem \cite{wh} implies that $(y_0,t_0)$ is a regular point whenever
\[\lim_{t\rightarrow t_0}\int \rho_{y_0,t_0}d\mu_t\leq 1+\epsilon,\]
contradicting the initial assumption that $(y_0,t_0)$ is a singular point.

\subsection{Convergence to a biholomorphic isometry}

In the preceding sections we have shown that the mean curvature flow $\Sigma_t$ of the graph of symplectomorphism $f:\mathbb{CP}^n\rightarrow \mathbb{CP}^n$ exists smoothly for all $t>0$, and that $\Sigma_t$ is a graph of symplectomorphisms for each $t$ under the pinching condition. We conclude the proof of Theorem~\ref{theorem} by showing that $\Sigma_t$ converges to the graph of a biholomorphic isometry.

By Proposition~\ref{evolution}:
\[\left(\frac{d}{dt}-\Delta\right)*\Omega=*\Omega \left[Q(\lambda_i,h_{jkl})+\displaystyle\sum_{k\text{ odd}}\frac{(1-\lambda_k^2)^2}{(1+\lambda_k^2)^2}\right]\]
along the mean curvature flow, where $Q\geq 0$ whenever $\frac{1}{\Lambda_0}\leq \lambda_i\leq \Lambda_0$.

We use this result to derive the evolution equation of $\ln *\Omega$, which we then apply to show that $\displaystyle\lim_{t\rightarrow\infty}*\Omega=\frac{1}{2^n}$.

\begin{pro} \label{logevolution}
Under the same assumption as in Proposition \ref{evolution},
at each point $q\in\Sigma_t$, $\ln *\Omega$ satisfies the following equation:
\begin{align*}
\frac{d}{dt}\ln *\Omega=&\Delta \ln *\Omega+\overline{Q}(\lambda_i, h_{jkl}) +\displaystyle\sum_{k}\displaystyle\sum_{i \neq k}\frac{\lambda_i}{(1+\lambda_k^2)(\lambda_i+\lambda_{i'})}(R_{ikik}-\lambda_k^2\tilde{R}_{ikik}),
\end{align*}
where $R_{ijkl}$ and $\tilde{R}_{ijkl}$ are the coefficients of the curvature tensors of $M$ and $\tilde{M}$ with respect to the chosen bases (\ref{e_i}) and (\ref{e_2n+i}), $i'=i+(-1)^{i+1}$, and
\begin{equation}
\overline{Q}(\lambda_i, h_{jkl})=Q(\lambda_i, h_{jkl})+\displaystyle\sum_{k}\left[\displaystyle\sum_{i\,\, \text{odd}}(\lambda_i-\lambda_{i'})h_{i i' k}\right]^2
\end{equation}
with $Q(\lambda_i, h_{jkl})$ given by Proposition~\ref{evolution} and equation (\ref{quadform}).
\end{pro}

\begin{proof} We compute
\[\frac{d}{dt}\ln *\Omega=\frac{1}{*\Omega}\frac{d}{dt} *\Omega \text{ and }
\Delta (\ln *\Omega)=\frac{*\Omega\Delta(*\Omega)-|\nabla *\Omega|^2}{(*\Omega)^2}.\]
By Proposition~\ref{evolution}, it follows that
\[
\left(\frac{d}{dt}-\Delta\right)\ln *\Omega = Q(\lambda_i,h_{jkl})+\displaystyle\sum_{k}\displaystyle\sum_{i \neq k}\frac{\lambda_i}{(1+\lambda_k^2)(\lambda_i+\lambda_{i'})}(R_{ikik}-\lambda_k^2\tilde{R}_{ikik})+\frac{|\nabla *\Omega|^2}{(*\Omega)^2}.
\]

We compute
\begin{align*}
(*\Omega)_k &=\displaystyle\sum_i\Omega(e_1, \ldots,(\nabla^{M\times\tilde{M}}_{e_k}-\nabla^\Sigma_{e_k})e_i,\ldots, e_{2n})\\
&=\displaystyle\sum_i\Omega(e_1,\ldots,\langle\nabla^{M\times\tilde{M}}_{e_k}e_i, \mathcal{J}e_p \rangle\mathcal{J}e_p,\ldots, e_{2n})\\
&=\displaystyle\sum_{p,i}\Omega(e_1, \ldots,\mathcal{J}e_p,\ldots, e_{2n})h_{pik}\\
\end{align*}

As the simplification of the expression $\mathcal{A}$ in the proof of Proposition \ref{evolution}, we obtain
\[(*\Omega)_k=*\Omega\displaystyle\sum_{i}(-1)^i\lambda_ih_{ii'k}=-*\Omega \sum_{i\,\,\text{odd}}(\lambda_i-\lambda_{i'})h_{i i'k}.\]
It follows that:
\begin{align*}
\frac{|\nabla *\Omega|^2}{(*\Omega)^2}&=\displaystyle\sum_{k}\left[\displaystyle\sum_{i\,\, \text{odd}}(\lambda_i-\lambda_{i'})h_{i i' k}\right]^2,\\
\end{align*}
and thus
\[\left(\frac{d}{dt}-\Delta\right)\ln *\Omega=\overline{Q}(\lambda_i,h_{jkl})+\displaystyle\sum_{k}\displaystyle\sum_{i \neq k}\frac{\lambda_i}{(1+\lambda_k^2)(\lambda_i+\lambda_{i'})}(R_{ikik}-\lambda_k^2\tilde{R}_{ikik}),\]
where $\overline{Q}(\lambda_i, h_{jkl})=Q(\lambda_i, h_{jkl})+\displaystyle\sum_{k}\left[\displaystyle\sum_{i\,\, \text{odd}}(\lambda_i-\lambda_{i'})h_{i i' k}\right]^2$ is a new quadratic form in $h_{ijk}$, with coefficients depending on the singular values of $f$.
\end{proof}

\begin{cor} \label{cpnlog}
Under the same assumption as in Proposition \ref{evolution}, suppose in addition that $M$ and $\tilde{M}$ are both $\cp^n$ with the Fubini-Study metric, then:
\begin{align*}
\frac{d}{dt}\ln *\Omega=&\Delta \ln *\Omega+\overline{Q}(\lambda_i, h_{ijk})+\sum_{k\text{ odd}}\frac{(1-\lambda_k^2)^2}{(1+\lambda_k^2)^2}.
\end{align*}
\end{cor}

\begin{proof}
This is a direct consequence of Proposition~\ref{logevolution} and Corollary~\ref{cpn}.
\end{proof}

\begin{rem}
$\overline{Q}$ is a positive definite quadratic form of $h_{ijk}$ whenever $Q$ is, and in fact it allows for an improvement of the pinching constant.
\end{rem}

We use the evolution equation of $\ln *\Omega$ to show that $\displaystyle\lim_{t\rightarrow\infty}*\Omega=\frac{1}{2^n}$. Fix a $k$ and notice that
\[\frac{(1-\lambda_k^2)^2}{(1+\lambda_k^2)^2}=\frac{(\lambda_k-\lambda_{k'})^2}{(\lambda_k+\lambda_{k'})^2}=\frac{x-4}{x},
\]
where $x=(\lambda_k+\lambda_{k'})^2$.

Since $\lambda_k\lambda_{k'}=1$, it follows that $\lambda_k+\lambda_{k'}\geq 2$, and thus $x\geq 4$. Moreover, the pinching condition implies that $x\leq \left({\Lambda_0}+\frac{1}{{\Lambda_0}}\right)^2$.

We claim
\[\frac{x-4}{x}\geq c\left(\frac{1}{2}\ln x-\ln2\right)\]
for $c=\frac{8}{({\Lambda_0}+\frac{1}{{\Lambda_0}})^2}$.

To see this, let $f(x)=\frac{x-4}{x}$, $g(x)=c\left(\frac{1}{2}\ln x-\ln 2\right)$ and notice that
$f(4)=g(4)=0$. We compute
\[f'(x)=\frac{x-x+4}{x^2}=\frac{4}{x^2} \text{ and } g'(x)=\frac{c}{2x}.
\]

Thus
\[\frac{f'(x)}{g'(x)}=\frac{4}{x^2}\frac{2x}{c}=\frac{8}{cx}\geq 1.
\]
The last inequality follows from the choice of $c$ and the fact that $x\leq \left({\Lambda_0}+\frac{1}{{\Lambda_0}}\right)^2$. Now since $f(4)=g(4)$ and $f'(x)\geq g'(x)$ for $4 \leq x \leq\left({\Lambda_0}+\frac{1}{{\Lambda_0}}\right)^2$, it follows that $f(x)\geq g(x)$.

Substituting back, we obtain\[\frac{(\lambda_k-\lambda_{k'})^2}{(\lambda_k+\lambda_{k'})^2}\geq c\left(\ln (\lambda_k+\lambda_{k'})-\ln 2\right),\]
and thus
\begin{align*}
\displaystyle\sum_{k\text{ odd}}\frac{(1-\lambda_k^2)^2}{(1+\lambda_k^2)^2}=\displaystyle\sum_{k\text{ odd}}\frac{(\lambda_k-\lambda_{k'})^2}{(\lambda_k+\lambda_{k'})^2}&\geq c\left(-\ln \displaystyle\prod_{k\text{ odd}}\frac{1}{\lambda_k+\lambda_{k'}} -n\ln 2\right)\\
&=-c\left(\ln *\Omega - \ln \frac{1}{2^n}\right).
\end{align*}
Therefore under the pinching condition:
\[\left(\frac{d}{dt}-\Delta\right)\left(\ln *\Omega- \ln \frac{1}{2^n}\right)\geq -c\left(\ln *\Omega - \ln \frac{1}{2^n}\right).\]

The pinching condition holds along the mean curvature flow, so this holds for all times. By the comparison principle for parabolic equations, $\displaystyle\lim_{t\rightarrow \infty}\displaystyle\min_{\Sigma_t}\ln *\Omega-\ln\frac{1}{2^n}=0$, and thus  $\displaystyle\lim_{t\rightarrow \infty}\displaystyle\min_{\Sigma_t} *\Omega=\frac{1}{2^n}$. This in turn implies, by Lemma~\ref{boundpinch},  that $\lambda_i\rightarrow 1 $ as $t\rightarrow\infty$ for all $i$.

For the rest of the proof, we modify the method from \cite{wa3} to show the second fundamental form is uniformly bounded in time. Let $\epsilon>0$ and let $\eta_\epsilon=*\Omega-\frac{1}{2^n}+\epsilon$. Note that $\displaystyle\min_{\Sigma_t}\eta_\epsilon$ is nondecreasing, and $\eta_\epsilon\rightarrow \epsilon$ when $t\rightarrow\infty$. Let $T_\epsilon\geq 0$ be a time such that $\eta_\epsilon|_{T_\epsilon} > 0$ (so that for all $t\geq T_\epsilon$: $\eta_\epsilon>0$).

Now for all $p\in M$, and all $t> T_\epsilon$:
\begin{align*}
\frac{d}{dt}\eta_\epsilon&=\Delta\eta_\epsilon +*\Omega(Q+B)\\
&\geq \Delta\eta_\epsilon +\delta *\Omega |\textrm{II}|^2\\
&=\Delta\eta_\epsilon +\frac{\delta}{\eta_\epsilon}\eta_\epsilon *\Omega |\textrm{II}|^2.
\end{align*}

On the other hand, from \cite{wa3}, $|\textrm{II}|^2$ satisfies the following equation along the mean curvature flow:
\begin{align*}
\frac{d}{dt}|\textrm{II}|^2&=\Delta |\textrm{II}|^2-2|\nabla \textrm{II}|^2+[(\nabla ^M_{\partial_k})\mathcal{R}(\mathcal{J}e_p, e_i,e_j,e_k)+(\nabla ^M_{\partial_j}\mathcal{R})(\mathcal{J}e_p,e_k,e_i,e_k)]h_{p i j}\\
&-2\mathcal{R}(e_l,e_i,e_j,e_k)h_{p l k}h_{p i j}+4 \mathcal{R}(\mathcal{J}e_p,\mathcal{J}e_q,e_j, e_k)h_{q i k}h_{p i j}\\
&-2\mathcal{R}(e_l,e_k,e_i,e_k)h_{p l j}h_{p i j}+\mathcal{R}(\mathcal{J}e_p,e_k,\mathcal{J}e_q,e_k)h_{q i j}h_{p i j}\\
&+\displaystyle\sum_{p, r, i, m} (\displaystyle\sum_{k} h_{p i k}h_{r m k}-h_{p m k}h_{r i k})^2+\displaystyle\sum_{i,j,m,k}(\displaystyle\sum_{p} h_{p i j}h_{p m k})^2.
\end{align*}

Since $M\times\tilde{M}$ is a symmetric space, the curvature tensor $\mathcal{R}$ of $M\times\tilde{M}$ is parallel, and thus $|\textrm{II}|^2$ satisfies:
\[\frac{d}{dt}|\textrm{II}|^2\leq \Delta|\textrm{II}|^2-2|\nabla \textrm{II}|^2+K_1|\textrm{II}|^4+K_2|\textrm{II}|^2\]
for positive constants $K_1$ and $K_2$ that depend only on $n$.

Therefore:
\begin{align*}
\frac{d}{dt}(\eta_\epsilon^{-1} |\textrm{II}|^2) &\leq -\eta_\epsilon^{-2}|\textrm{II}|^2(\Delta\eta_\epsilon +\delta *\Omega |\textrm{II}|^2)+\eta_\epsilon^{-1} (\Delta|\textrm{II}|^2-2|\nabla II|^2+K_1|\textrm{II}|^4+K_2|\textrm{II}|^2)\\
&=-\eta_\epsilon^{-2}\Delta\eta_\epsilon|\textrm{II}|^2+\eta_\epsilon^{-1}\Delta |\textrm{II}|^2 -2\eta_\epsilon^{-1}|\nabla \textrm{II}|^2+\eta_\epsilon^{-2}(\eta_\epsilon K_1-\delta *\Omega)|\textrm{II}|^4+\eta_\epsilon^{-1}K_2|\textrm{II}|^2\\
&=\Delta(\eta_\epsilon^{-1})|\textrm{II}|^2-2\eta_\epsilon^{-3}|\nabla \eta_\epsilon|^2|\textrm{II}|^2+\eta_\epsilon^{-1}\Delta |\textrm{II}|^2 -2\eta_\epsilon^{-1}|\nabla \textrm{II}|^2\\
&+\eta_\epsilon^{-2}(\eta_\epsilon K_1-\delta *\Omega)|\textrm{II}|^4+\eta_\epsilon^{-1}K_2|A|^2\\
&=\Delta(\eta_\epsilon^{-1})|\textrm{II}|^2-2\eta_\epsilon|\nabla (\eta_\epsilon^{-1})|^2|\textrm{II}|^2+\eta_\epsilon^{-1}\Delta |\textrm{II}|^2 -2\eta_\epsilon^{-1}|\nabla \textrm{II}|^2\\
&+\eta_\epsilon^{-2}(\eta_\epsilon K_1-\delta *\Omega)|\textrm{II}|^4+\eta_\epsilon^{-1}K_2|\textrm{II}|^2\\
&=\Delta(\eta_\epsilon^{-1} |\textrm{II}|^2)-2\nabla(\eta_\epsilon^{-1})\cdot\nabla (|\textrm{II}|^2)-2\eta_\epsilon|\nabla (\eta_\epsilon^{-1})|^2|\textrm{II}|^2 -2\eta_\epsilon^{-1}|\nabla \textrm{II}|^2\\
&+\eta_\epsilon^{-2}(\eta_\epsilon K_1-\delta *\Omega)|\textrm{II}|^4+\eta_\epsilon^{-1}K_2|\textrm{II}|^2.\\
\end{align*}
We apply the relation that \[-2\nabla(\eta_\epsilon^{-1})\cdot\nabla (|\textrm{II}|^2)-2\eta_\epsilon|\nabla (\eta_\epsilon^{-1})|^2|\textrm{II}|^2=-2\eta_\epsilon \nabla(\eta_\epsilon^{-1})\cdot \nabla (\eta_\epsilon^{-1}|\textrm{II}^2).\]
Therefore the function $\psi=\eta_\epsilon^{-1}|\textrm{II}|^2$ satisfies:
\begin{align*}
\frac{d}{dt}\psi &\leq \Delta \psi -2\eta_\epsilon\nabla \eta_\epsilon^{-1}\cdot \nabla \psi +(\eta_\epsilon K_1-\delta *\Omega)\psi^2+K_2\psi\\
&\leq \Delta \psi -2\eta_\epsilon\nabla \eta_\epsilon^{-1}\cdot \nabla \psi +(\epsilon K_1-\delta C_0)\psi^2+K_2\psi,
\end{align*}
where $C_0=\displaystyle\min_{\Sigma_0} *\Omega$, since $\displaystyle\min_{\Sigma_t} *\Omega$ is nondecreasing and $\eta_\epsilon\leq \epsilon$. $\epsilon$ can be chosen small enough so that $\epsilon K_1-\delta C_0<0$.
Then by the comparison principle for parabolic PDE,
$\psi\leq y(t)$ for all $t\geq T_\epsilon$, where $y(t)$ is the solution of the ODE
\[\frac{d}{dt}y=-(\delta C_0-\epsilon K_1)y^2+K_2y
\]satisfying the initial condition $y(T_\epsilon)=\displaystyle\max_{\Sigma_{T_\epsilon}}\psi$. $y(t)$ can be solved explicitly:
\[y(t)=\begin{cases}
\frac{K_2}{\delta C_0-\epsilon K_1}, & \mbox{if } \displaystyle\max_{\Sigma_{T_\epsilon}}\psi=\frac{K_2}{\delta C_0-\epsilon K_1}\\
\frac{K_2 K e^{K_2 t}}{(\delta C_0-\epsilon K_2)Ke^{K_2 t}-1}, & \mbox{otherwise}
\end{cases},
\]
where $K$ is a constant satisfying $K>0$ if $\displaystyle\max_{\Sigma_{T_\epsilon}}\psi>\frac{K_2}{\delta C_0-\epsilon K_1}$, and $K<0$ if $\displaystyle\max_{\Sigma_{T_\epsilon}}\psi<\frac{K_2}{\delta C_0-\epsilon K_1}$. Thus
\[|\textrm{II}|^2\leq \eta_\epsilon y(t)\leq \epsilon y(t)
\] for all $t\geq T_\epsilon$.

Sending $t\rightarrow\infty$ and $\epsilon\rightarrow 0$, we conclude that $\displaystyle\max_{\Sigma_t} |\textrm{II}|^2\rightarrow 0$ as $t\rightarrow \infty$. Finally, the induced metric and the volume functional both have analytic dependence on $F$, so by Simon's theorem \cite{si} the flow converges to a unique limit at infinity.

Since $\lambda_i\rightarrow 1$ for all $i$ as $t\rightarrow\infty$, the limit map is an isometry. Denote it by $f_{\infty}$. Being symplectic is a closed property, so $f_\infty$ is symplectic. Then at every $p\in M$:
\[Df_\infty J=\tilde{J}Df_\infty\]
The same is true for the inverse of $f_\infty$, and thus the map $f_\infty$ is biholomorphic.

\section{Appendix}

\subsection{Proof of Lemma \ref{evalue}}
We recall that $h_{ijk}$ is symmetric in all three indexes, that all indexes range from $1$ to $2n$ unless otherwise (such as $i$ odd) is mentioned, and that $i'=i+(-1)^{i+1}$.
The object of study is the quadratic form $\tilde{Q}(h_{ijk})$ given by
\begin{equation}\begin{split} \label{quadform2}
&\displaystyle\sum_{i,j,k}h_{ijk}^2-2\displaystyle\sum_{k}\displaystyle\sum_{i\text{ odd}}(h_{iik}h_{i'i'k}-h_{ii'k}^2)+4\displaystyle\sum_{k}\displaystyle\sum_{i\text{ odd}<j\text{ odd}}(h_{i'jk}h_{j'ik}-h_{ijk}h_{j'i'k})\\
&=A+B+C.
\end{split}
\end{equation}

 We shall use the full symmetry of $h_{ijk}$ to show the smallest eigenvalue of $\tilde{Q}$ is positive.  The quadratic form $\tilde{Q}$ will be divided into three summands such that the indexes of the first summand $\tilde{Q}_1$ only involve $i$ and $i'$ for odd $i$'s, the indexes of the second summand $\tilde{Q}_2$ only involve $i, i', j, j'$ for odd $i$ and odd $j$ with $i\not= j$, the indexes of the third summand $\tilde{Q}_3$ involve $i, i', j, j', k, k'$ for odd $i$, $j$, and $k$ such that no two of them are the same. This corresponds to a direct sum decomposition of the space of $h_{ijk}$ in which each of the summand is an invariant subspace of the symmetry group. We state the result in two Lemmas and give the proof of second Lemma  first, which implies Lemma \ref{evalue}. In the rest of the section, we verify the formulas in first Lemma.

\begin{lem} The three summands of $\tilde{Q}$ in (\ref{quadform2}) can be rewritten in the following way:
\[\begin{split}
A&=\sum_{i} h_{iii}^2+ 3\sum_{i \text{ odd}} (h_{ii'i'}^2+h_{i'ii}^2)\\
&+3 \sum_{i\text{ odd} < j \text{ odd}}(h_{ijj}^2+h_{ij'j'}^2+h_{i'jj}^2+h_{i'j'j'}^2+h_{jii}^2+h_{j'ii}^2+h_{ji'i'}^2+h_{j'i'i'}^2)\\
&+6\sum_{i \text{ odd}<j \text{ odd}} (h_{ii'j}^2+ h_{ii'j'}^2+h_{ijj'}^2+h_{i'jj'}^2)\\
&+6\sum_{i \text{ odd} < j \text{ odd} <k \text{ odd}}( h_{ijk}^2+ h_{ijk'}^2+h_{ij'k}^2+ h_{ij'k'}^2+h_{i'jk}^2+ h_{i'jk'}^2+h_{i'j'k}^2+ h_{i'j'k'}^2)\\
B&=-2\sum_{ i\text{ odd}}h_{iii}h_{i'i'i}+2\sum_{ i\text{ odd}} h_{ii'i}^2-2\sum_{ i\text{ odd}}h_{iii'}h_{i'i'i'}+2\sum_{ i\text{ odd}} h_{ii'i'}^2,\\
&-2\sum_{ i\text{ odd}< j \text{ odd} }(h_{iij}h_{i'i'j}-h_{ii'j}^2+h_{iij'}h_{i'i'j'}-h_{ii'j'}^2)\\
&-2\sum_{ i\text{ odd}< j \text{ odd}}(h_{jji}h_{j'j'i}-h_{jj'i}^2+h_{jji'}h_{j'j'i'}-h_{jj'i'}^2), \text{ and}\\
C&=4\sum_{i\text{ odd}<j\text{ odd}}(h_{i'ji}h_{j'ii}-h_{iji}h_{j'i'i}+h_{i'ji'}h_{j'ii'}-h_{iji'}h_{j'i'i'})\\
&+4\sum_{i\text{ odd}<j\text{ odd}}(h_{i'jj}h_{j'ij}-h_{ijj}h_{j'i'j}+h_{i'jj'}h_{j'ij'}-h_{ijj'}h_{j'i'j'})\\
&+4\sum_{i \text{ odd}< j\text{ odd}<k\text{ odd}}(h_{j'k i}h_{k'ji}-h_{jki}h_{k'j'i}+h_{j'k i'}h_{k'j i'}-h_{jk i'}h_{k'j'i'})\\
&+4\sum_{i\text{ odd}< j \text{ odd}< k\text{ odd}}(h_{i'kj}h_{k'ij}-h_{ikj}h_{k'i'j}+h_{i'kj '}h_{k'ij'}-h_{ikj'}h_{k'i'j'})\\
&+4\sum_{i\text{ odd}<j\text{ odd}< k\text{ odd}}(h_{i'jk}h_{j'ik}-h_{ijk}h_{j'i'k}+h_{i'jk'}h_{j'ik'}-h_{ijk'}h_{j'i'k'}).\\\end{split}\]
\end{lem}

\begin{lem} $\tilde{Q}=\tilde{Q}_1+\tilde{Q}_2+\tilde{Q}_3$ where
$\tilde{Q}_1$ is the sum over all odd indexes $i$ of \[ h_{iii}^2+h_{i'i'i'}^2+ 5 (h_{ii'i'}^2+ h_{i'ii}^2) -2h_{iii}h_{i'i'i}-2h_{iii'}h_{i'i'i'},\]

$\tilde{Q}_2$ is the sum over all indexes $(i, j)$ with, ${i\text{ odd} < j \text{ odd}}$,  of
\[\begin{split}
&3 (h_{ijj}^2+h_{ij'j'}^2+h_{i'jj}^2+h_{i'j'j'}^2+h_{jii}^2+h_{j'ii}^2+h_{ji'i'}^2+h_{j'i'i'}^2)\\
&+8 (h_{ii'j}^2+ h_{ii'j'}^2+h_{ijj'}^2+h_{i'jj'}^2)-2(h_{iij}h_{i'i'j}+h_{iij'}h_{i'i'j'})-2(h_{jji}h_{j'j'i}+h_{jji'}h_{j'j'i'})\\
&+4(h_{i'ji}h_{j'ii}-h_{iji}h_{j'i'i})+4(h_{i'ji'}h_{j'ii'}-h_{iji'}h_{j'i'i'})\\
&+4(h_{i'jj}h_{j'ij}-h_{ijj}h_{j'i'j})+4(h_{i'jj'}h_{j'ij'}-h_{ijj'}h_{j'i'j'}),\\
\end{split}\]
and $\tilde{Q}_3$ is the sum over all indexes $(i, j, k)$ with, ${i \text{ odd} < j \text{ odd} <k \text{ odd}}$,  of
\[\begin{split}
&6( h_{ijk}^2+ h_{ijk'}^2+h_{ij'k}^2+ h_{ij'k'}^2+h_{i'jk}^2+ h_{i'jk'}^2+h_{i'j'k}^2+ h_{i'j'k'}^2)\\
&+4(h_{j'k i}h_{k'ji}-h_{jki}h_{k'j'i}+h_{j'k i'}h_{k'j i'}-h_{jk i'}h_{k'j'i'})\\
&+4(h_{i'kj}h_{k'ij}-h_{ikj}h_{k'i'j}+h_{i'kj '}h_{k'ij'}-h_{ikj'}h_{k'i'j'})\\
&+4(h_{i'jk}h_{j'ik}-h_{ijk}h_{j'i'k}+h_{i'jk'}h_{j'ik'}-h_{ijk'}h_{j'i'k'}).\\\end{split}\]

In addition, the following inequalities hold:
\[\tilde{Q}_1\geq \sum_{i \text{ odd}}(3-\sqrt{5}) (h_{iii}^2+h_{i'i'i'}^2+ h_{ii'i'}^2+ h_{i'ii}^2).\]
\[\tilde{Q}_2\geq 2\sum_{i\text{ odd} < j \text{ odd}} (h_{ijj}^2+h_{ij'j'}^2+h_{i'jj}^2+h_{i'j'j'}^2+h_{jii}^2+h_{j'ii}^2+h_{ji'i'}^2+h_{j'i'i'}^2+h_{ii'j}^2+ h_{ii'j'}^2+h_{ijj'}^2+h_{i'jj'}^2).\]
\[\tilde{Q}_3\geq 4\sum_{i \text{ odd} < j \text{ odd} <k \text{ odd}}( h_{ijk}^2+ h_{ijk'}^2+h_{ij'k}^2+ h_{ij'k'}^2+h_{i'jk}^2+ h_{i'jk'}^2+h_{i'j'k}^2+ h_{i'j'k'}^2).\]
Thus,  \[\tilde{Q}(h_{ijk})\geq (3-\sqrt{5}) ||h_{ijk}||^2\] where \[||h_{ijk}||^2=\sum_{i} h_{iii}^2+\sum_{i\not= j} h_{ijj}^2+\sum_{i<j<k} h_{ijk}^2.\]
\end{lem}
\begin{proof}
For each odd $i$, the expression in $\tilde{Q}_1$ can be further divided into two identical quadratic forms of two variables, each has smallest eigenvalue $3-\sqrt{5}$.
For each index $(i, j)$ with ${i\text{ odd} < j \text{ odd}}$, the expression in $\tilde{Q}_2$ can be further divided into four identical quadratic forms of three variables, each has smallest eigenvalue $2$. For each index $(i, j, k)$ with ${i \text{ odd} < j \text{ odd} <k \text{ odd}}$, the expression in $\tilde{Q}_3$ can be further divided into two identical quadratic forms of four variables, each has smallest eigenvalue $4$.
\end{proof}

First of all,
\begin{equation}\label{eq_A}A=\sum_{i} h_{iii}^2+3\sum_{i<j} h_{ijj}^2+3\sum_{i<j} h_{jii}^2+6\sum_{i<j<k} h_{ijk}^2.\end{equation}

Write \[\sum_{i<j} h_{ijj}^2=\sum_{i \text{ odd }<j} h_{ijj}^2+\sum_{i \text{ even }<j \text{ odd }} h_{ijj}^2+\sum_{i \text{ even }<j \text{ even}} h_{ijj}^2.\] In the first summand, it is possible that $j$ equals $i'$, thus
\begin{equation}\label{i<j}\sum_{i<j} h_{ijj}^2=\sum_{i \text{ odd}} h_{ii'i'}^2+\sum_{i \text{ odd }<j \text{ odd}} (h_{ijj}^2+ h_{ij'j'}^2)+\sum_{i \text{ odd }<j \text{ odd }} h_{i'jj}^2+\sum_{i \text{ odd }<j \text{ odd}} h_{i'j'j'}^2.\end{equation}
Similarly,
\begin{equation}\label{j<i}\sum_{i<j} h_{jii}^2=\sum_{i \text{ odd}} h_{i'ii}^2+\sum_{i \text{ odd }<j \text{ odd}} (h_{jii}^2+ h_{j'ii}^2)+\sum_{i \text{ odd }<j \text{ odd }} h_{ji'i'}^2+\sum_{i \text{ odd }<j \text{ odd}} h_{j'i'i'}^2.\end{equation}

On the other hand,
\begin{equation}\label{i<j<k}\begin{split}&\sum_{i<j<k} h_{ijk}^2\\
&=\sum_{i \text{ odd}<j<k} h_{ijk}^2+\sum_{i \text{ even} <j<k} h_{ijk}^2\\
&=\sum_{i \text{ odd} <k, i'<k} h_{ii'k}^2+\sum_{i \text{ odd}<j<k, j\not= i'} h_{ijk}^2+\sum_{i \text{ even} <j \text{ odd} <k} h_{ijk}^2+\sum_{i \text{ even} <j \text{ even} <k} h_{ijk}^2.\\\end{split}\end{equation}

The first term on the right hand side of (\ref{i<j<k}) equals
\[ \sum_{i \text{ odd} < j \text{ odd}} (h_{ii'j}^2+h_{ii'j'}^2). \]
The second term on the right hand side of (\ref{i<j<k}) equals \[\sum_{i \text{ odd}<j<k, j\not= i'} h_{ijk}^2=\sum_{i \text{ odd}<j \text{ odd} <k} h_{ijk}^2+\sum_{i \text{ odd}<j \text{ even} <k, j\not= i'} h_{ijk}^2.\] It is possible for $k$ to equal to $j'$ in the first summand, thus, the second term is
\[\sum_{i \text{ odd}<j \text{ odd}} h_{ijj'}^2+\sum_{i \text{ odd}<j \text{ odd} <k \text{ odd}} (h_{ijk}^2+ h_{ijk'}^2+h_{ij'k}^2+ h_{ij'k'}^2).\]
The third term on the right hand side of (\ref{i<j<k}) equals
\[\sum_{i \text{ even} <j \text{ odd} <k} h_{ijk}^2=\sum_{i \text{ odd} <j \text{ odd} } h_{i'jj'}^2+\sum_{i \text{ odd} <j \text{ odd} <k \text{ odd}} (h_{i'jk}^2+h_{i'jk'}^2).\]
The fourth term on the right hand side of (\ref{i<j<k}) equals \[\sum_{i \text{ even} <j \text{ even} <k} h_{ijk}^2=\sum_{i \text{ odd} <j \text{ odd} <k \text{ odd}} (h_{i'j'k}^2+h_{i'j'k'}^2).
\]
Therefore,
\begin{equation}\label{i<j<k_2}\begin{split}&\sum_{i<j<k} h_{ijk}^2\\
&=\sum_{i \text{ odd}<j \text{ odd}} (h_{ii'j}^2+ h_{ii'j'}^2+h_{ijj'}^2+h_{i'jj'}^2)\\
&+\sum_{i \text{ odd} < j \text{ odd} <k \text{ odd}}( h_{ijk}^2+ h_{ijk'}^2+h_{ij'k}^2+ h_{ij'k'}^2+h_{i'jk}^2+ h_{i'jk'}^2+h_{i'j'k}^2+ h_{i'j'k'}^2).\\\end{split}\end{equation}

Putting (\ref{i<j}), (\ref{j<i}), and (\ref{i<j<k_2}) into (\ref{eq_A}), we obtain the expression for $A$.

We proceed to compute $B$ and $C$ in the same manner.
\[\begin{split}
B&=-2\sum_{ i\text{ odd}}(h_{iii}h_{i'i'i}-h_{ii'i}^2)-2\sum_{ i\text{ odd}}(h_{iii'}h_{i'i'i'}-h_{ii'i'}^2)\\
&-2\sum_{ i\text{ odd}, j \text{ odd}, i\not= j }(h_{iij}h_{i'i'j}-h_{ii'j}^2+h_{iij'}h_{i'i'j'}-h_{ii'j'}^2)\\
&=-2\sum_{ i\text{ odd}}h_{iii}h_{i'i'i}+2\sum_{ i\text{ odd}} h_{ii'i}^2-2\sum_{ i\text{ odd}}h_{iii'}h_{i'i'i'}+2\sum_{ i\text{ odd}} h_{ii'i'}^2\\
&-2\sum_{ i\text{ odd}< j \text{ odd} }(h_{iij}h_{i'i'j}-h_{ii'j}^2+h_{iij'}h_{i'i'j'}-h_{ii'j'}^2)\\
&-2\sum_{ i\text{ odd}< j \text{ odd}}(h_{jji}h_{j'j'i}-h_{jj'i}^2+h_{jji'}h_{j'j'i'}-h_{jj'i'}^2).
\end{split}\]

\[\begin{split}
C&=4\sum_{i\text{ odd}<j\text{ odd}}(h_{i'ji}h_{j'ii}-h_{iji}h_{j'i'i})+4\sum_{i\text{ odd}<j\text{ odd}}(h_{i'ji'}h_{j'ii'}-h_{iji'}h_{j'i'i'})\\
&+4\sum_{i\text{ odd}<j\text{ odd}}(h_{i'jj}h_{j'ij}-h_{ijj}h_{j'i'j})+4\sum_{i\text{ odd}<j\text{ odd}}(h_{i'jj'}h_{j'ij'}-h_{ijj'}h_{j'i'j'})\\
&+4\sum_{i\text{ odd}<j\text{ odd}}\left[\sum_{k \text{ odd}, k\not= i, j}(h_{i'jk}h_{j'ik}-h_{ijk}h_{j'i'k}+h_{i'jk'}h_{j'ik'}-h_{ijk'}h_{j'i'k'})\right],\end{split}\] while
\[\begin{split}&\sum_{i\text{ odd}<j\text{ odd}}\left[\sum_{k \text{ odd}, k\not= i, j}(h_{i'jk}h_{j'ik}-h_{ijk}h_{j'i'k}+h_{i'jk'}h_{j'ik'}-h_{ijk'}h_{j'i'k'})\right]\\
=&\sum_{k \text{ odd}< i\text{ odd}<j\text{ odd}}(h_{i'jk}h_{j'ik}-h_{ijk}h_{j'i'k}+h_{i'jk'}h_{j'ik'}-h_{ijk'}h_{j'i'k'})\\
&\sum_{i\text{ odd}< k \text{ odd}< j\text{ odd}}(h_{i'jk}h_{j'ik}-h_{ijk}h_{j'i'k}+h_{i'jk'}h_{j'ik'}-h_{ijk'}h_{j'i'k'})\\
&\sum_{i\text{ odd}<j\text{ odd}< k\text{ odd}}(h_{i'jk}h_{j'ik}-h_{ijk}h_{j'i'k}+h_{i'jk'}h_{j'ik'}-h_{ijk'}h_{j'i'k'})\\
=&\sum_{i \text{ odd}< j\text{ odd}<k\text{ odd}}(h_{j'k i}h_{k'ji}-h_{jki}h_{k'j'i}+h_{j'k i'}h_{k'j i'}-h_{jk i'}h_{k'j'i'})\\
&\sum_{i\text{ odd}< j \text{ odd}< k\text{ odd}}(h_{i'kj}h_{k'ij}-h_{ikj}h_{k'i'j}+h_{i'kj '}h_{k'ij'}-h_{ikj'}h_{k'i'j'})\\
&\sum_{i\text{ odd}<j\text{ odd}< k\text{ odd}}(h_{i'jk}h_{j'ik}-h_{ijk}h_{j'i'k}+h_{i'jk'}h_{j'ik'}-h_{ijk'}h_{j'i'k'}).\\ \end{split}\]

\bibliographystyle{plain}

\end{document}